\documentclass[11pt,leqno]{article}
\usepackage{mymacros}
\usepackage{physics}
\usepackage{cite,color}
\usepackage[normalem]{ulem}
\usepackage{tikz}
\usetikzlibrary{arrows.meta, positioning}

\usepackage{tocloft}

\newtheorem{question}{Question}

\newcommand{\PP}{\bm{P}}

\newcommand{\wick}[1]{\mathopen{:}#1\mathclose{:}}

\newcommand{\Eg}{\mathsf{E}}

\newcommand{\1}{{\bf 1}}
\newcommand{\hShG}{{\rm hShG}}
\newcommand{\hGFF}{{\rm hGFF}}
\newcommand{\eff}{{\rm eff}}
\renewcommand{\var}{\operatorname{Var}}

\newcommand{\eps}{\epsilon}
\newcommand{\bZ}{\mathbb{Z}}
\newcommand{\bR}{\mathbb{R}}

\title{Mass gap for the hierarchical sinh-Gordon model}
\author{Omar Abdelghani\footnote{Courant Institute of Mathematical Sciences, New York University. E-mail: {\tt oa2391@nyu.edu}.}
  \and Roland Bauerschmidt\footnote{Courant Institute of Mathematical Sciences, New York University. E-mail: {\tt bauerschmidt@cims.nyu.edu}.}
 \footnote{Current address: Institut f\"ur Angewandte Mathematik, Universit\"at Bonn. E-mail: {\tt bauerschmidt@uni-bonn.de}.}
  \and Michael Hofstetter\footnote{Universit\"at Wien. E-mail: {\tt michael.hofstetter@univie.ac.at}.}
  \and Ofer Zeitouni\footnote{Weizmann Institute. E-mail: {\tt ofer.zeitouni@weizmann.ac.il}.}}
\date{September 15, 2026}

\begin{document}
\setcounter{page}{0}
\newpage

\maketitle
\begin{abstract}
  The (massless) sinh-Gordon model is defined in terms of a continuum massless Gaussian free field perturbed by a cosh interaction.
  For the hierarchical version of the model, we prove existence of the infinite volume limit, a uniform log-Sobolev inequality, and a mass gap.
  Our results hold for all $b^2 \in (0,1)$ and simplify for $b^2 \in (0,1/2)$ for which the Gaussian multiplicative chaos has two moments. 
  In an appendix, our renormalization group analysis is compared with the conjectures and controversies in physics which are mostly based on formal
  analytic continuation of conjectures for the sine-Gordon model and we raise some
  questions. 
\end{abstract}

\setcounter{tocdepth}{2}
\tableofcontents

\section{Introduction and main results}

\subsection{Introduction}

Despite its convex action, the definition and behavior of the sinh-Gordon model
are surprisingly confusing and subtle, and remain debated even in physics.
This paper clarifies some aspects from the probabilistic point of view---at least in the hierarchical version of the model.

From the probabilistic perspective, the sinh-Gordon model should be defined by the formal path integral measure
\begin{equation} \label{e:shG-cont1}
  \text{``}\exp\qa{\frac{1}{4\pi} \int_{\R^2} \varphi (\Delta \varphi) \,dx - \frac{\mu}{4b^2} \int_{\R^2} \wick{\cosh(2b \varphi)} \, dx} \, d\varphi\text{''}
\end{equation}
where the first term corresponds to the Gaussian free field normalized such that its covariance satisfies $C(x,y) \sim \log 1/|x-y|$ as $|x-y|\to 0$,
and the parameter ranges $0 < b^2 < 1/2$ and $0 < b^2 < 1$ correspond to the $L^2$ and $L^1$ regimes, respectively, of the Gaussian multiplicative chaos
$\wick{\cosh(2b\varphi)}$, while $b^2=1$ corresponds to the ``self-dual point''.
There is no explicit mass term in \eqref{e:shG-cont1} and the massless Gaussian free field by itself cannot be defined on the full plane (only up to constants).
For $b^2 \geq 1$ there is a debate in physics about a zero mode that might have to be properly included.
As discussed later, this zero mode is not relevant in the massive regime $b^2<1$ that we focus on.

The notation $\wick{\cdot}$ denotes a multiplicative renormalization and, from the probabilistic perspective,
the measure \eqref{e:shG-cont1} should be obtained as the $\epsilon \to 0$
and $\Lambda \to \R^2$ limit of its regularized version in which $\varphi_\epsilon$ is some regularization of $\varphi$ at distances $\epsilon$ and
the $\wick{\cosh}$ term is interpreted as
\begin{equation} \label{e:shG-cont2}
  \frac{\mu_\epsilon}{4b^2} \int_\Lambda  \cosh(2b\varphi_\epsilon) \, dx, \qquad \mu_\epsilon = \mu \epsilon^{2b^2},
\end{equation}
which, as discussed below, is the correct choice for $b^2 < 1$.
%

The variant of the sinh-Gordon model with an additional explicit mass term $\frac12 m^2 \varphi^2$
is a version of the H{\o}egh-Krohn model (or ``massive'' sinh-Gordon model)
and is very well understood for all $b^2 \in (0,1)$; see \cite{MR0489552,MR356761} and \cite{MR4528973,2512.18927} for recent references.
Even more elementary is the (massless or massive) model on the unit lattice, where no counterterm is needed
and which can thus be analyzed in a straightforward way for any $b^2>0$ by methods for uniformly convex potentials such as the Bakry--\'Emery criterion,
the Helffer--Sj\"ostrand representation, or the Brydges--Fr\"ohlich--Spencer random walk representation.
On the other hand, there is no full construction of the continuum massless model, and its behavior is debated even in physics.
Even though $\cosh$ is a uniformly convex function, the strict convexity of the interaction is lost as $\epsilon \to 0$
due to the multiplicative counterterm.
Nonetheless, it is natural to expect that the limiting continuum model has a strictly positive physical mass,
in the sense of correlation decay and a uniform log-Sobolev constant---in analogy with its lattice version.

In fact, the continuum model is expected to be integrable
and explicit conjectures for the mass and the one-point functions have been proposed for $b^2 \in (0,1)$.
These conjectures are not based on a direct analysis of the path integral \eqref{e:shG-cont1} and rather follow by combination
of various nonrigorous methods such as formal analytic continuation of
conjectures for the sine-Gordon model and deformation of Liouville conformal field theory, as well as consistency with the semiclassical limit $b \to 0$;
see \cite{MR4258290} for a review.
Rigorous progress has been made on the $S$-matrix bootstrap program for the sinh-Gordon model \cite{MR4680395,MR4607722},
but the relation to the probabilistic construction of the sinh-Gordon model remains completely unclear.
The different analytic structure of the sine- and sinh-Gordon interactions has also led to some doubts on the conjectures
in the region $1/2 \leq b^2 < 1$, see \cite[Section~2.2.7]{MR4258290} and Appendix~\ref{app:physics} for a discussion.
The correct definition and behavior of the model for $b^2\geq 1$ remains a topic of an interesting debate.
From a mathematical perspective, the analytic continuation from $\cos$ to $\cosh$ interaction seems completely unclear for any $b$,
and a probabilistic construction that uses the convexity of $\cosh$ would seem natural and desirable.

Despite recent mathematical progress \cite{2408.16574,2408.16649,MR5055747}, there is no complete probabilistic construction of the sinh-Gordon model on $\R^2$ yet.
Hierarchical models have a long history in the context of the renormalization group, see \cite{MR693402,MR3269690} for reviews
and \cite{MR3969983} for an introduction with emphasis on parallels with the Euclidean models.
Recent works also include \cite{Vilas}.
Not all quantum field theories are accurately captured by their hierarchical analogs,
but in superrenormalizable scalar theories with finite field strength renormalization
(as the sinh-Gordon model for $b^2\in (0,1)$)
they are expected and in many cases confirmed to be excellent prototypes.
Hierarchical models, which appear there as branching random walks or multiplicative cascades, also have provided important stepping stones
in the study of the Gaussian multiplicative chaos \cite{MR829798}, which is an important ingredient in our analysis 
of the hierarchical sinh-Gordon model.

\subsection{Branching random walk and hierarchical free field}
\label{sec:hierGFF}

In the hierarchical sinh-Gordon model, the Gaussian free field is replaced by its hierarchical version -- a bi-infinite branching random walk (BRW)
whose covariance operator is the inverse of the hierarchical Laplacian $\Delta_H$ introduced now.
Further details are included in Appendix~\ref{app:hier}.

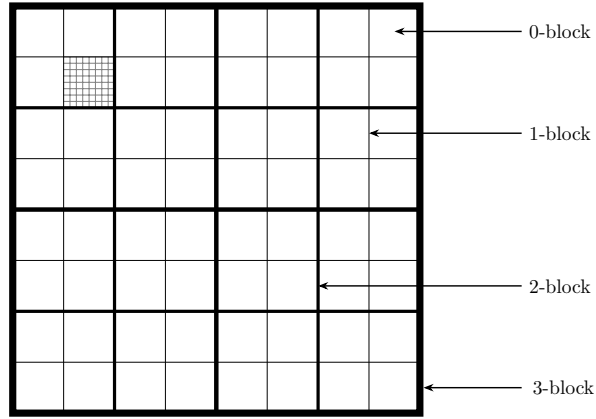
\begin{figure}
  \begin{center}\resizebox{8cm}{!}{
\begin{tikzpicture} 

\draw[step=0.125cm,gray,very thin] (1,6) grid (2,7);
\draw[line width=4pt] (0,0) -- (8,0) -- (8,8) -- (0,8) -- cycle;

\draw[line width = 2.5pt] (4,0) -- (4,8);
\draw[line width = 2.5pt] (0,4) -- (8,4);

\draw[line width = 1.8pt] (0,2) -- (8,2);
\draw[line width = 1.8pt] (0,6) -- (8,6);
\draw[line width = 1.8pt] (2,0) -- (2,8);
\draw[line width = 1.8pt] (6,0) -- (6,8);

\draw (1,0) -- (1,8);
\draw (3,0) -- (3,8);
\draw (5,0) -- (5,8);
\draw (7,0) -- (7,8);

\draw (0,1) -- (8,1);
\draw (0,3) -- (8,3);
\draw (0,5) -- (8,5);
\draw (0,7) -- (8,7);

\draw[thick,-{Stealth[length=2mm]}] (10,7.5) -- (7.5,7.5) node[right=4cm, anchor=east ] {$0$-block};
\draw[thick,-{Stealth[length=2mm]}] (10,5.5) -- (7,5.5) node[right=4.5cm, anchor=east ] {$1$-block};
\draw[thick,-{Stealth[length=2mm]}] (10,2.5) -- (6,2.5) node[right=5.5cm, anchor=east ] {$2$-block};
\draw[thick,-{Stealth[length=2mm]}] (10,0.5) -- (8.07,0.5) node[right=3.5cm, anchor=east ] {$3$-block};

\end{tikzpicture}
}\end{center}
  \caption{Illustration of the continuum hierarchical lattice with $d=2$ and $L=2$.
    Each $0$-block (unit block) contains nested $j$-blocks, where $j=-1,\dots,-N$,
    and is contained in arbitrarily many larger $j$-blocks, where $j=1,\dots,M$.
    The continuum limit corresponds to $-N\to-\infty$
    and the infinite volume limit to $M\to\infty$.
    This block division of $\R^d$ is equivalent to the tree associated with a branching random walk with branching rate $L^d$.
    The nodes at height $j$ are the scale-$j$ blocks.
    The branching random walk itself is defined in terms of independent Gaussian random variables of variance $L^{-dj}\gamma_j \approx L^{-(d-2)j}$ when $d=2$
    on the nodes of the tree (blocks).
    The massive version has increments with variances $L^{-dj} \gamma_j(m) \approx L^{-(d-2)j}(1+O(m^2L^{2j}))^{-1}$, see \eqref{e:gammam}.
    \label{fig:hier}}
\end{figure}


Given a fixed integer $L > 1$ whose value is not important (and one could for example choose $L=2$ for concreteness),
for every $L$-adic scale $j\in \Z$,
the space $\R^d$ is divided into blocks $B = L^j(x + [0,1)^d)$ where $x\in \Z^d$,
see Figure~\ref{fig:hier}. The set of scale-$j$ blocks is denoted $\cB_j$.
The hierarchical Laplacian and its inverse can be obtained by replacing the Fourier basis in the representation of the usual Laplacian
by essentially the basis of Haar wavelets associated with these blocks.
In the notation of \cite[Section~4]{MR3969983} and Appendix~\ref{app:hier},
the hierarchical Laplacian acting on $f: \R^d \to \R$ has the spectral representation:
\begin{equation}
  - \Delta_H = c_L \sum_{j \in \Z} \lambda_j P_j, \qquad P_j=Q_{j-1} - Q_{j}, \qquad \lambda_j = L^{-2j},
\end{equation}
where the $Q_j$ are block averaging operators onto the scale-$j$ blocks,
i.e., $Q_j f(x)$ is equal to the average of $f$ over the $j$-block containing $x$,
and the $P_j$ are orthogonal projections onto mean $0$ functions on the $j$-blocks.
The constant $c_L$ is chosen in \eqref{eq:var-normalisation} below to match the Euclidean normalization.
The system of projections $P_j$ forms a projection-valued measure on $L^2(\bR^d)$ by an elementary application of the dominated convergence theorem.
This is a hierarchical Fourier decomposition of $\Delta_H$ with $\lambda_j$ analogous to the Fourier multiplier $\lambda(k) = |k|^2$ associated with the Euclidean Laplacian.
Correspondingly, the resolvent of $\Delta_H$ has the spectral decomposition
\begin{equation}
\label{eq:CH-m-spectral}
  (-\Delta_H+m^2)^{-1} = \sum_{j\in \Z} (c_L\lambda_j+m^2)^{-1} P_j
  = \sum_{j\in \Z} \gamma_j(m) Q_{j},
\end{equation}
where (with the last estimate for $m^2 L^{2j} \ll 1$)
\begin{equation} \label{e:gammam}
  \gamma_j(m) = (c_L\lambda_{j+1}+m^2)^{-1}- (c_L\lambda_{j}+m^2)^{-1}
  =c_L^{-1} L^{2j} (L^2-1+O(m^2L^{2j}))
  .
\end{equation}
To match the Euclidean normalization, see \eqref{e:GreenH-log} below, we now choose
\begin{equation}
\label{eq:var-normalisation}
c_L =(L^2-1)/\log L,
\end{equation}
where $\log$ is the natural logarithm, so that $\gamma_j(m)=L^{2j} (\log L)(1+O(m^2 L^{2j}))$.
For $x,y \in \bR^d$, $x\neq y$, it follows from the above that
\begin{equation}
  (-\Delta_H+m^2)^{-1}(x,y) = \sum_{j(x,y) \leq j} L^{-dj}\gamma_j(m)
\end{equation}
where $j(x,y)$ is the first $j \in \Z$ such that $x,y$ are in the same $j$-block.
Let $d_H(x,y) = L^{j(x,y)}$. Then,
for $d>2$ the sum is bounded above and below by $L$-dependent multiples of $d_H(x,y)^{-(d-2)}$ as $m\to 0$.
For $d=2$ the sum is log-divergent as $m\to 0$ but, as $d_H(x,y) \to 0$,
\begin{equation} \label{e:GreenH-log}
  (-\Delta_H+m^2)^{-1}(x,y)
  = \sum_{j(x,y) \leq j}  L^{-dj}\gamma_j(m) 
	= \log\frac{1}{m d_H(x,y)} + O(1).
\end{equation}

The covariance of the hierarchical Gaussian free field with mass $m>0$ would be $(-\Delta_H+m^2)^{-1}$.
The massless version of the Gaussian free field (or its hierarchical version) is not defined on $\R^2$,
only modulo constants. We will only need the regularized version of the hierarchical massless free field,
given for short-distance (or ultraviolet) regularization scale $-N$ (eventually tending to $-\infty$)
and long-distance (or infrared) regularization scale $M$ (eventually tending to $+\infty$) by
\begin{equation} \label{e:CNM}
  C_{-N,M} = \sum_{j=-N}^M \gamma_j Q_j, \qquad \gamma_j= \gamma_j(0) = L^{2j} \log L.
\end{equation}
General results or explicit construction in terms of independent standard Gaussian random variables imply that there is a unique Gaussian measure
$\nu^{\hGFF(-N,M)}$ with covariance $C_{-N,M}$ on $\mathcal{S}'(\bR^2)$ equipped with the cylinder-set $\sigma$-algebra. This is the regularized hierarchical GFF
whose expectation is denoted $\Eg_{C_{-N,M}}$  and  which we  use to define the sinh-Gordon model.

Let us remark that a random distribution $\varphi\sim \nu^{\hGFF(-N,M)}$ is almost surely equal to a piecewise function on the $-N$-blocks.
In other words, $\nu^{\hGFF(-N,M)}$ is the pushforward of a measure on $\bR^{(L^{-N}\bZ^2)}$. Throughout this paper we will freely make this identification without further comment.

\begin{remark}
The choice of $c_L$ in \eqref{eq:var-normalisation} is made precisely so that the centered Gaussian measure with covariance $C_{-N,0}$ has $\var \zeta(x) = (\log L)(N+1)$.
In particular, the expectation of $V_{-N-1}(\Lambda,\varphi)$, defined in \eqref{e:VN} below,
under the covariance $C_{-N,0}$ is independent of $N$
when $\Lambda$ is a scale-$0$ block. Indeed, it equals $\mu/4b^2$ for every $N$.
\end{remark}

\subsection{Hierarchical sinh-Gordon model}

Except that the Gaussian free field is replaced by its hierarchical version,
the hierarchical sinh-Gordon model is defined as in the Euclidean model \eqref{e:shG-cont1}--\eqref{e:shG-cont2}.
Namely, starting from ultraviolet regularization scale $\epsilon=L^{-N-1}$ and $\Lambda \subset \R^2$ bounded,
 the regularized sinh-Gordon model has interaction potential
given for $\varphi: \R^2 \to \R$ that is piecewise constant by
\begin{equation}
  V_{\epsilon}(\Lambda, \varphi) = \frac{\mu_{\epsilon}}{4b^2} \int_\Lambda  \cosh(2b\varphi(x)) \,dx, \qquad \mu_{\epsilon}= \mu \epsilon^{2b^2}.
\end{equation}
The unimportant choice of $-N-1$ rather than $-N$ as the ultraviolet scale will turn out convenient.
More explicitly, if $\varphi$ is constant on the blocks of $\cB_{-N}$,
\begin{align} \label{e:VN}
  V_{\epsilon}(\Lambda,\varphi)
  &= V_{-N-1}(\Lambda,\varphi)
  \nnb
  &= \frac{\mu_{-N-1}}{4b^2}  L^{-dN} \sum_{x\in (L^{-N}\Z)^2 \cap \Lambda} \cosh(2b\varphi(x)), \qquad \mu_{-N-1}= \mu L^{-2b^2(N+1)}.
\end{align}
Denote by $\Eg_C[\cdot]$ the expectation with respect to the centered Gaussian measure with covariance $C$,
and recall the covariance of the regularized hierarchical Gaussian free field from \eqref{e:CNM}.

\begin{definition} \label{defn:nu-NM}
For $b>0$ and $\mu>0$ and ultraviolet scale $-N>-\infty$ and infrared scale $M<+\infty$,
the regularized hierarchical sinh-Gordon measure $\nu^{\hShG(b,\mu|-N,M)}$ on $\cS'(\R^2)$ is defined by the expectation value
\begin{equation} \label{e:nu-NM}
\avg{F}_{\hShG(b,\mu|-N,M)} \propto \Eg_{C_{-N,M}} \qa{e^{-V_{-N-1}(\Lambda,\varphi)} F(\varphi)}
\end{equation}
for any bounded $F: \cS'(\R^2) \to \R$ that depends on the field $\varphi$ in a bounded union of $M$-blocks $\Lambda \subset \R^2$.
\end{definition}

The definition makes sense because, by the hierarchical structure of $C_{-N,M}$,
the measure on the right-hand side is, in fact, a product measure over the $M$-blocks of $\Lambda$.
Thus the measure can formally be extended to $\Lambda = \R^2$ as a product measure whose distribution agrees 
on any finite union of $M$-blocks with the measure on the right-hand side.
This product measure is clearly supported on distributions, and it is also not difficult to see that these are tempered. For the next theorem, we equip 
$\mathcal{S}'(\bR^2)$ with the strong dual topology. It is shown in \cite{biermé2017generalizedrandomfieldslevys} that the Borel $\sigma$-algebra for the weak and strong-dual topologies coincide and are equal to the cylinder-set $\sigma$-algebra.
We thus choose to define weak convergence of measures relative to the strong dual topology.

By hierarchical cylinder functions we mean functions $\varphi\mapsto F(\varphi(f_1),\dots,\varphi(f_n))$ where the functions $f_i$ are square-integrable and piecewise-constant on $j$-blocks for some $j$ and $F$ is $C^1$, bounded, and with bounded derivative.
By the same method that we use to prove existence of the Sinh-Gordon model, one may show that the map $f\mapsto \varphi(f)$ continuously extends to a process indexed by $L^2(\bR^2)$, so hierarchical cylinder functions are well-defined random variables. For a hierarchical cylinder function, we may define $\fdv{F}{\varphi(x)}=\sum_{i}f_i(x)\partial_i F(\varphi(f_1),\dots, \varphi(f_n))$. Thus $\grad F$ is a random variable with values in $L^2$. 

\begin{theorem} \label{thm:hshg}
  For any $b^2 \in (0,1)$ and $\mu>0$, 
  there is an effective coupling constant satisfying
  \begin{equation} \label{e:thm-mueff}
    0<\mu_{\rm eff}  = \mu_{\rm eff}(b,\mu) \approx \mu^{1/(1+b^2)},
  \end{equation}
  with implicit constant depending on $b \in (0,1)$, such that the following hold.
  
  \medskip
  \noindent
  (i) The probability measures $(\nu^{\hShG(b,\mu|-N,M)})$ on $\cS'(\R^2)$ converge weakly to a probability measure $\nu^{\hShG(b,\mu)}$ as first $N\to\infty$ (ultraviolet limit)
  and then $M\to\infty$ 
  (infrared limit)
  satisfying the following scale covariance:
  for $s\in L^\Z = \{\dots, L^{-1}, L^0, L^1, \dots\}$ with $R_s\varphi(x) = \varphi(x/s)$ and $R_sF(\varphi) = F(R_s\varphi)$,
  \begin{equation} \label{e:thm-scaling}
    \avg{F}_{\hShG(b,\mu)} = \avg{R_sF}_{\hShG(b,\mu s^{2+2b^2})}.
  \end{equation}

  \smallskip
  \noindent
  (ii) The associated maps $\phi_{-N,M}:\mathcal{S}(\bR^2)\to L^2(\mu_{\hShG(b,\mu|-N,M)})$
  and $\phi:\mathcal{S}(\bR^2)\to L^2(\mu_{\hShG(b,\mu)})$
  extend continuously to a bounded linear maps $\phi_{-N,M}:L^2(\bR^2)\to L^2(\mu_{\hShG(b,\mu|-N,M)})$, $\phi:L^2(\bR^2)\to L^2(\mu_{\hShG(b,\mu)})$.

  For all $f\in L^2(\bR^2)$, $\phi_{-N,M}(f)\to \phi(f)$ in distribution.

  \smallskip
  \noindent
  (iii)
  The measure $\nu^{{\rm hShG}(b,\mu)}$ has log-Sobolev constant of order $\mu_{\rm eff}$: for any hierarchical cylinder functional
  $F: \cS'(\R^2) \to [0,\infty)$,
  \begin{equation} \label{e:thm-LS}
    \ent_{\nu^{\hShG(b,\mu)}}(F) \lesssim \frac{2}{\mu_{\rm eff}} \E_{\nu^{\hShG(b,\mu)}}\qa{\|\nabla \sqrt{F}\|_{L^2}^2},
  \end{equation}
  where $\ent_{\nu}(F) = \E_{\nu}[F\log F]-\E_\nu[F]\log \E_\nu[F]$ denotes the relative entropy.

  \smallskip
  \noindent
  (iv)
  There is a uniform mass gap in the sense that
  for $x,y\in \R^2$ with $d_H(x,y) \geq 1$,
  \begin{equation} \label{e:thm-twopoint}
    \E_{\nu^{\hShG(b,\mu)}}\qb{\varphi(x)\varphi(y)} \lesssim (-\Delta_H+\mu_{\rm eff})^{-1}(x,y).
  \end{equation}
\end{theorem}

The above inequality is in the sense of distributions. Moreover the $L^2$-gradient in \eqref{e:thm-LS} requires a standard explanation.
For more elementary statements, the reader can also consult the proofs
which work uniformly with finite-dimensional regularizations.
The limiting versions are then consequences of stability under weak limits.

As a general consequence of the massive decay of the two-point function \eqref{e:thm-twopoint} and the FKG inequality
\cite[Section VIII.7]{MR0489552}, which both the hierarchical and Euclidean sinh-Gordon model satisfy,
see Appendix~\ref{app:hier},
it follows that there is also clustering of all truncated $n$-point correlation functions,
i.e., decay of cumulants.
For the Euclidean model, this would imply the mass gap in the Hamiltonian.
Note that massive correlation decay corresponds to an additional decay factor $\approx 1/(1+m^2d_H(x,y)^2)$
in the hierarchical model compared to the exponential factor $\approx e^{-m|x-y|}$ in the Euclidean model.

\subsection{Outline and interpretation}
\label{sec:interpretation}

The main ingredient of the proof is to show that the renormalized potential $V_j$,
defined and analyzed in Section~\ref{sec:renpot} below,
becomes uniformly convex with the estimate that one might expect from perturbation theory:
\begin{equation} \label{e:intro-Hebd}
  L^{2j} \He V_j \gtrsim L^{(2+2b^2)j} \mu,
\end{equation}
up to scales $j \leq j_0(b,\mu)$ such that the right-hand side is of order $1$. In fact, for $b^2 \in (0,1/2)$, we show
\begin{equation} \label{e:intro-HeL2}
  L^{2j} \He V_j \geq L^{2j} \He V_j(0) = L^{(2+2b^2)j} \mu - O(L^{(2+2b^2)j}\mu)^2,
\end{equation}
as well as $L^{2j} \He V_j \gtrsim 1$ for $j>j_0(b,\mu)$, and that these estimates are optimal.
For $b^2 \in [1/2,1)$, we expect the quadratic error to be false but prove that there is $p\in (1,2)$ depending on $b$ such that
\begin{equation} \label{e:intro-HeL1}
  L^{2j} \He V_j(0) = L^{(2+2b^2)j} \mu - O(L^{(2+2b^2)j}\mu)^p,
\end{equation}
and also $L^{2j} \He V_j \gtrsim 1$ for $j>j_0(b,\mu)$.

The convexification of the renormalized potential may at first sight not seem surprising, but
we think it is somewhat interesting.
From the renormalization group perspective, the mechanism is different from that in a $\varphi^4$ model, for example,
where convexification occurs at high temperature due to perturbation theory which is essentially a ``small field'' feature.
The convexification of the $\cosh$ interaction comes from fluctuations in the ``large field'' region. This is more
similar to what happens in the sine-Gordon model, but the unbounded $\cosh$ interaction requires a different approach
and the perturbative structure is different. For example, as we now outline, 
the analogue of the estimate \eqref{e:intro-HeL1} holds in the sine-Gordon model with quadratic error term throughout the subcritical phase.

To explain this, recall that the sine-Gordon model is defined in essentially the same way as the sinh-Gordon model, with interaction
\begin{equation}
  V_{\epsilon}(\varphi) = -\frac{z_\epsilon}{4\beta^2} \int \cos(2\beta \varphi_\epsilon) \, dx, \qquad z_\epsilon = \epsilon^{-2\beta^2} z,
\end{equation}
so that the sinh-Gordon model is formally obtained by taking $z=-\mu$ and $\beta = ib$.
Our normalization is such that the KT transition takes place at $\beta^2=1$ and the free fermion point is $\beta^2=1/2$ (for the Euclidean model).
The renormalization group flow of the renormalized coupling constant $z_j$ satisfies
\begin{equation}
  z_{j+1}= L^{2-2\beta^2} z_j + O(z_j^2),
\end{equation}
and one has $L^{2j} \He V_j(0) = z_j + O(z_j^2)$. In particular, for all $\beta^2 \in (0,1)$, one has
\begin{equation} \label{e:intro-SGHe}
  L^{2j} \He V_j(0) = L^{(2-2\beta^2) j} z + O(L^{(2-2\beta^2) j} z)^2.
\end{equation}
For the hierarchical version of the model,
this follows for example easily from the presentation in \cite[Section~4]{MR4061408},
 which is essentially an adaptation of the method of \cite{MR859369}, see also \cite{MR2523458}.
This estimate should be compared with \eqref{e:intro-HeL2} for $b^2 \in (0,1/2)$ and \eqref{e:intro-HeL1} for $b^2 \in [1/2,1)$.
The quadratic error bound \eqref{e:intro-SGHe} for all $\beta^2 \in (0,1)$ rather than only in the regime $\beta^2 \in (0,1/2)$ is different from what we expect
is true in the sinh-Gordon model.

For the sine-Gordon model the definition of the running coupling constant $z_j$ in a rigorous renormalization group framework
is well understood, in terms of a Fourier decomposition of the periodic potential (equivalent to a charge decomposition of the equivalent Coulomb gas),
again see \cite{MR4061408,MR2523458} for the hierarchical setting.
On the other hand, the meaning of the running coupling constant $\mu_j$ of the sinh-Gordon model is not obvious,
and we use the Hessian of the renormalized potential as a proxy that is defined non-perturbatively.

It might initially seem surprising that the renormalized potential
of the sine-Gordon model seems better behaved than that of the sinh-Gordon model
when the ``$L^2$ threshold'' of the GMC is crossed ($\beta^2 \geq 1/2$ respectively $b^2 \geq 1/2$).
Indeed, it is well known that the sine-Gordon model has an infinite energy renormalization for $\beta^2 \geq 1/2$
that results in
a limiting measure that is not locally absolutely continuous with respect to the hierarchical GFF when $\beta^2 \in [1/2,1)$, see for example \cite{2502.02554},
whereas the sinh-Gordon measure is locally absolutely continuous for all $b^2 \in (0,1)$.
Why is perturbation theory then  better behaved for the singular sine-Gordon model than for the ultraviolet finite sinh-Gordon model?
The infinite energy renormalization is not in contradiction with
the fact that the finite-volume sine-Gordon correlation functions are, in fact, analytic in $z$ (in a strip containing the real axis).
This is because the infinite energy renormalization for $\beta^2\in [1/2,1)$ is not visible in the correlation functions or the Hessian of the renormalized potential.
On the other hand, the failure of perturbation theory for the sinh-Gordon model is not in contradiction to absence of energy renormalization (and absolute continuity),
and essentially a consequence of the fact that the Gaussian multiplicative chaos $\wick{e^{\pm 2b\varphi}}$ has finitely many moments under the free field measure.
We include a more detailed discussion of this point in Appendix~\ref{app:sinsinhpert}.

Finally, in a different direction,
we mention that the effective potential of the sine-Gordon model continues to have the noncompact $\Z$-symmetry, $\varphi \mapsto \varphi+n\pi/\beta$, $n\in \Z$.
It is expected that the massless measure exists in infinite volume in the sense of spontaneous symmetry breaking (with massive correlation decay),
but the log-Sobolev inequality is \emph{not} expected to hold (as more generally in situations of spontaneously broken symmetry).

\subsection{Notation}

The notations $\lesssim$ and $\approx$ denote inequalities up to constants, where
implicit constants can depend on $b$ and $L$ but no other parameters.

\section{Convexity of the renormalized potential}
\label{sec:renpot}

In this section, we prove our main technical estimates, which are strong convexity estimates for the renormalized
potential of the hierarchical sinh-Gordon model.

Using the conventions of Section~\ref{sec:hierGFF}, we consider the decomposition of the hierarchical Gaussian free field
truncated at small distance (ultraviolet) scale $-N$ and large distance (infrared) scale $j$:
\begin{equation} \label{e:C-decomp}
  C_j = C_{-N,j} = \sum_{k=-N}^j \gamma_k Q_k, \qquad \text{with $\gamma_k = L^{2k} \log L$},
\end{equation}
for $j\geq -N$, and $C_{-N-1}=0$.
This covariance is equivalent to that of a branching random walk on the leaves at height $N+j+1$ with standard Gaussian increments of variance $\log L$.
Denote by $\cB_j$ the set of $j$-blocks $L^j(x+[0,1)^2)$, $x\in \Z^2$,
and by $\cB_j(\Lambda)$ the set of $j$-blocks contained in $\Lambda$.
The renormalized potential is defined by
\begin{equation} \label{e:Vj}
  V_j(\varphi) = V_j(\Lambda,\varphi) = -\log \Eg_{C_j}[e^{-V_{-N-1}(\Lambda,\varphi+\zeta)}] = \sum_{B\in \cB_j(\Lambda)} V_j(B,\varphi|_B),
\end{equation}
where $V_{-N-1}(\Lambda,\varphi)$ is defined in \eqref{e:VN}
and $\varphi: \bR^2 \to \R$ is a locally bounded external field.

The same definition would apply for the Euclidean model, with the hierarchical covariance decomposition replaced by, for example, a heat kernel decomposition of the usual Euclidean free field.
The simplifications of the hierarchical model are that we only need
to consider fields $\varphi$ constant on $j$-blocks $B$ and that the renormalized potential $V_j$ is additive over $j$-blocks
as indicated in \eqref{e:Vj},
where each $V_j(B,\varphi)$ can be restricted to a function of a single block variable $\varphi$ constant on $B$.
%
Thus while $V_j$ is the global renormalized potential depending on $\varphi: \R^2 \to \R$, all information is contained in the local
renormalized potentials $V_j(B,\cdot)$.

\begin{definition}
  The space $X_j \subset L^2(\R^2)$ consists of $\varphi: \R^2 \to \R$
  which are constant on $j$-blocks, i.e., $\varphi$ is constant on the block $L^j(x+[0,1)^2)$ for each $x \in \Z^2$.
\end{definition}

\begin{proposition} \label{prop:GHS}
  $\He V_j(\varphi) \geq \He V_j(0)$ for any $\varphi \in X_j$. 
\end{proposition}

\begin{proof}
  Since $V_j$ is additive over $j$-blocks and $\varphi$ is constant on $j$-blocks, we can regard $\varphi$ as constant
  and restrict to a single block that is omitted from the notation.
  Thus with minor abuse of notation, each $V_{j}$ is a function of a single variable.
  The initial potential $V_{-N-1}(\varphi) =\frac{\mu_{-N-1}}{4b^2}\cosh(2b\varphi)$ is symmetric and $V_{-N-1}'''(\varphi) >0$ for $\varphi>0$.
  This property is clearly preserved by sums (and multiplication with positive constants).
  It is also preserved by log-Gaussian convolutions, as a special case of the GHS inequality \cite[Theorem~1.2.c]{MR395659}.
  In more detail, given a scale-$j$ block $B$ and $\varphi \in X_j$, the previous scale
  renormalized potential on $B$ splits as
  $V_{j-1}(B,\varphi) = \sum_{b \in \cB_{j-1}(B)} V_{j-1}(b,\varphi) = L^{2} V_{j-1}(b,\varphi)$
  for any $b \in \cB_{j-1}(B)$.
  Since the covariance $C_j-C_{j-1} =\gamma_j Q_j$ is supported on $X_j$ and using Girsanov's formula,
  \begin{align}
    V_j(B,\varphi)
    &= -\log \Eg_{C_{j}-C_{j-1}}\q{e^{-V_{j-1}(B,\varphi+\zeta)}}
    \nnb
    &= \frac{(\varphi,\varphi)_B}{2\gamma_j} -\log \Eg_{C_{j}-C_{j-1}}\qa{e^{-L^{2}V_{j-1}(\zeta,b)} e^{\frac{1}{\gamma_j}(\zeta,\varphi)_B}},
  \end{align}
  where $(\cdot,\cdot)_B$ is the $L^2$ inner product restricted to $B$.
  Since $\varphi$ is constant on $B$,
  we have $(\zeta,\varphi)_B = \varphi (\zeta,1)$.
  Thus, the $\varphi$-derivative of the second term in the last display is proportional to minus the expectation of
  the magnetisation $(\zeta,1)$ under the measure associated with the expectation on the right-hand side
  and the GHS inequality implies
  $V_j'''(\varphi) \geq 0$ for $\varphi>0$ if $V_{j-1}'''(\varphi,b) \geq 0$ for $\varphi>0$.
\end{proof}
In the remainder of this section we prove the following propositions:
\begin{proposition} \label{prop:H0}
Let $b^2\in (0,1/2)$. Then for $j \in \Z$ such that $L^{(2+2b^2)j}\mu \ll 1$ independent of $N$, 
\begin{equation} \label{e:H0}
  L^{2j} \He V_j(0) \geq L^{(2+2b^2)j} \mu  - O(L^{(2+2b^2)j}\mu)^2.
\end{equation}
\end{proposition}

\begin{proposition} \label{prop:H0-L1}
Let $b^2\in [1/2,1)$. Then there is $p\in (1,2)$ depending on $b^2$ such that for all $j \in \Z$ such that $L^{(2+2b^2)j}\mu \ll 1$ independent of $N$, 
\begin{equation} \label{e:H0-L1}
  L^{2j} \He V_j(0) \geq L^{(2+2b^2)j} \mu  - O(L^{(2+2b^2)j}\mu)^p.
\end{equation}
\end{proposition}

The effective scale $j_0(b,\mu)$ is defined such that the right-hand side is of order one. It thus satisfies $L^{(2+2b^2)j_0(b,\mu)}\mu \approx 1$
and $L^{2j_0(b,\mu)} \approx \mu^{-1/(1+b^2)}$.

\begin{corollary} \label{cor:Hej0}
  Let $b^2 \in (0,1)$. Then $L^{2j} \He V_j(0) \gtrsim 1$ for all $j \geq j_0(b,\mu)$.
\end{corollary}

There is also a simple upper bound (which essentially follows from the Cram\'er--Rao inequality)
that shows that the above two bounds are sharp for $b^2 \in (0,1/2)$.

\begin{proposition}  \label{prop:CR}
For any $b >0$ and $\mu >0$,
\begin{equation}
  L^{2j} \He V_j(0) \leq \frac{L^{(2+2b^2)j}\mu}{1+(\log L)L^{(2+2b^2)j}\mu}.
\end{equation}
More generally, for all $\varphi \in X_j$,
\begin{equation}
  L^{2j} \He V_j(\varphi) \leq \frac{L^{(2+2b^2)j}\mu \cosh(2b\varphi)}{1+(\log L)L^{(2+2b^2)j}\mu \cosh(2b\varphi)},
\end{equation}
where $\cosh(2b\varphi)$ is viewed as a $j$-block diagonal matrix (or by additivity of $V_j$ one can take $\varphi$ constant).
\end{proposition}

\subsection{General convexity estimates on the renormalized potential}

Unless stated otherwise, throughout this section $\varphi \in X_{-N}$.
Let
\begin{equation}
  H_j(\varphi) = \He V_j(\varphi).
\end{equation}
The Hessian of the renormalized potential at the ultraviolet scale is given by
the diagonal operator
\begin{equation}
  H_{-N-1}(\varphi) = \He V_{-N-1}(\varphi) = \diag (\mu_{-N-1}\cosh(2b\varphi(x)))_{x\in\Lambda_{-N}}.
\end{equation}
Define the fluctuation measure $\mu_{k,j}^\varphi$ from scale $k$ to $j$, with $k<j$:
\begin{equation} \label{e:Pkj-mukj-def}
  \PP_{k,j} F(\varphi)
  = \E_{\mu_{k,j}^\varphi}[F(\zeta)]
  = \frac{\Eg_{C_j-C_k}[e^{-V_{k}(\varphi+\zeta)} F(\varphi+\zeta)]}{\Eg_{C_j-C_k}[e^{-V_{k}(\varphi+\zeta)}]}
  ,
\end{equation}
and when the index $k$ is omitted we imply that $k=-N-1$ is the regularization scale.

The following lemma bounds the Hessian of the renormalized potential in terms of the bare one.
The subsequent lemma bounds the Hessian at one scale in terms of that
of the previous scale.

\begin{lemma}
\label{lem:BL}
As a quadratic form, for any $\varphi \in X_{-N}$,
\begin{equation} \label{e:H-BL}
  C_j^{1/2} H_j(\varphi) C_j^{1/2}
  \geq
  \PP_{-N-1,j}\qa{ \frac{C_j^{1/2} H_{-N-1}C_j^{1/2}}{1+C_j^{1/2}H_{-N-1}C_j^{1/2}}}(\varphi).
\end{equation}
\end{lemma}
\begin{proof}
%
This estimate follows from \cite[Section~3.4]{MR4798104}, but for convenience, we recall the above application of the Brascamp--Lieb inequality.
Indeed, since
\begin{equation}
  V_j(\varphi) = -\log \Eg_{C_j} \qa{e^{-V_{-N-1}(\varphi+\zeta)}},
\end{equation}
the Hessian is given by
\begin{equation}
  (f,H_j(\varphi)f) = (f,\He V_j(\varphi)f) = \E_{\mu_{-N-1,j}^\varphi}[(f,H_{-N-1}(\zeta)f)] - \var_{\mu_{-N-1,j}^\varphi}\qa{(\nabla V_{-N-1}(\zeta),f)}.
\end{equation}
Thus
\begin{equation}
  (f, C_j^{1/2} H_j(\varphi)C_j^{1/2}f) 
  = \E_{\mu_{-N-1,j}^\varphi}\qa{(f,C_j^{1/2}H_{-N-1}(\zeta)C_j^{1/2}f)} - \var_{\mu_{-N-1,j}^\varphi}\qa{(C_j^{1/2}\nabla V_{-N-1}(\zeta),f)}
  .
\end{equation}
Let $\hat F(\zeta) = (C_j^{1/2}\nabla V_{-N-1}(\zeta),f)$ and $\hat H(\zeta) = C_j^{1/2} H_{-N-1}(\zeta)C_j^{1/2}$. Then
\begin{equation}
  (C_j^{1/2}\nabla \hat F, g) = (f, C_j^{1/2} \He V_{-N-1}(\zeta) C_j^{1/2} g) = (f,\hat H g),
\end{equation}
and, by the Brascamp--Lieb inequality, using that $C_j$ is invertible on $X_{-N}$ and $\He V_{-N-1} \geq 0$,
\begin{align}
  \var_{\mu_{-N-1,j}^\varphi}\q{\hat F}
  &\leq \E_{\mu_{-N-1,j}^\varphi}\qa{(\nabla \hat F, (C_j^{-1}+ \He V_{-N-1})^{-1} \nabla \hat F)}
    \nnb
  &= \E_{\mu_{-N-1,j}^\varphi}\qa{(\nabla \hat F, C_j^{1/2} (1+  C_j^{1/2}\He V_{-N-1}C_j^{1/2})^{-1}C_j^{1/2} \nabla \hat F)}.
    \nnb
  &= \E_{\mu_{-N-1,j}^\varphi}\qa{(f,\hat H (1+ \hat H)^{-1} \hat H f)}.
\end{align}
Thus we obtain, as quadratic forms,
\begin{equation}
  C_j^{1/2}H_j(\varphi)C_j^{1/2}
  \geq \E_{\mu_{-N-1,j}^\varphi}\qa{\hat H - \hat H(1+\hat H)^{-1} \hat H}
  = \E_{\mu_{-N-1,j}^\varphi}\qa{\frac{\hat H}{1+\hat H}},
\end{equation}
which is the claimed bound.
\end{proof}

\begin{lemma}
 As a quadratic form, for any $\varphi \in X_j$ and $j>-N$,
\begin{equation} \label{e:BL-large}
  \gamma_j Q_jH_j(\varphi)Q_j
  \geq
  \PP_{j-1,j}
  \qa{\frac{\gamma_jQ_j H_{j-1} Q_j}{1+\gamma_j Q_j H_{j-1}Q_j}}(\varphi)
  .
\end{equation}
\end{lemma}

\begin{proof}
The proof is very similar to the previous one. We start from
\begin{equation}
  (f, Q_j H_j(\varphi)Q_j f)
  = \E_{\mu_{j-1,j}^\varphi}\qa{(f,Q_jH_{j-1}(\zeta)Q_jf)} - \var_{\mu_{j-1,j}^\varphi}\qa{(Q_j\nabla V_{j-1}(\zeta),f)}
  .
\end{equation}
Let $\hat F(\zeta) = (Q_j\nabla V_{j-1}(\zeta),f)$ and $\hat H(\zeta) = Q_j H_{j-1}(\zeta)Q_j$. Then
\begin{equation}
  (Q_j\nabla \hat F, g) = (f, Q_j \He V_{j-1}(\zeta) Q_j g) = (f,\hat H g),
\end{equation}
and, by the Brascamp--Lieb inequality,  
\begin{align}
  \var_{\mu_{j-1,j}^\varphi}\q{\hat F}
  &\leq \E_{\mu_{j-1,j}^\varphi}\qa{(\nabla \hat F, (\gamma_j^{-1} Q_j+ \He V_{j-1})^{-1} \nabla \hat F)}
    \nnb
  &= \gamma_j  \E_{\mu_{j-1,j}^\varphi}\qa{(\nabla \hat F, Q_j (1+ \gamma_j Q_j\He V_{j-1}Q_j)^{-1}Q_j \nabla \hat F)}.
    \nnb
  &= \gamma_j  \E_{\mu_{j-1,j}^\varphi}\qa{(f,\hat H (1+\gamma_j \hat H)^{-1} \hat H f)}.
\end{align}
Thus
\begin{equation}
  \gamma_j Q_jH_j(\varphi)Q_j
  \geq \gamma_j \E_{\mu_{j-1,j}^\varphi}\qa{\hat H - \gamma_j\hat H(1+\gamma_j\hat H)^{-1} \hat H}
  ,
\end{equation}
which is the claim.
\end{proof}

\subsection{Convexity for $b^2 \in (0,1/2)$: Proof of Proposition~\ref{prop:H0}}

The proof of Proposition~\ref{prop:H0} uses the following very crude bound which is simply Jensen's inequality.
We consider a single $j$-block $B$, $j=-N,\ldots, M$, and will often abbreviate its side length by $\ell= L^j$.

\begin{lemma}
  For any $b>0$ and $\mu>0$, and any fixed $j$-block $B$ with $|B|=L^{2j}=\ell^2$ and $\varphi\in X_{-N}$:
  \begin{equation} \label{e:VB-lbub}
    0 \geq V_j(B,0)-V_j(B,\varphi) \geq -O(\mu \ell^{2b^2})\int_{B} \cosh(2b\varphi(y)) \, dy.
  \end{equation}
\end{lemma}

\begin{proof}
  By definition,
  \begin{equation} \label{e:VN-bis}
    V_{-N-1}(B,\varphi)
    = \frac{\mu_{-N-1}}{4b^2} \int_B \cosh(2b\varphi(x)), \qquad \mu_{-N-1}= \mu L^{-2b^2(N+1)},
  \end{equation}
  and $e^{-V_{j}(B,\varphi)} = \Eg_{C_j}[e^{-V_{-N-1}(B,\zeta+\varphi)}]$ by the definition \eqref{e:Vj}.
  Using $V_{-N-1} \geq 0$ for the lower bound and Jensen's inequality for the upper bound, 
  it follows that
  \begin{equation} \label{e:VB-Jensen}
    0 \leq V_j(B,\varphi) \leq O(\mu L^{2b^2j})\int_{B} \cosh(2b\varphi(y)) \, dy.
  \end{equation}
  In more detail, the upper bound is:
  \begin{equation}
    e^{-V_{j}(B,\varphi)}
    =
    \Eg_{C_j}[e^{-V_{-N-1}(B,\zeta+\varphi)}]
    \geq
    \exp\pB{-\Eg_{C_j}[V_{-N-1}(B,\zeta+\varphi)]}
  \end{equation}
  where
  \begin{equation} \label{e:VN-expect}
    \Eg_{C_j}[V_{-N-1}(B,\zeta+\varphi)]
    = \Eg_{C_j} \qa{ \int_{B} \frac{\mu}{4b^2}\epsilon^{2b^2} \cosh(2b(\zeta+\varphi)) \, dx}
    = \frac{\mu}{4b^2} L^{2b^2j} \int_{B}  \cosh(2b\varphi) \, dx.
  \end{equation}
  This proves \eqref{e:VB-Jensen}.
  The lower bound in \eqref{e:VB-lbub} immediately follows. The upper bound in \eqref{e:VB-lbub}
  holds because $X_{-N} \ni \varphi \mapsto V_j(B,\varphi)$
  is symmetric and strictly convex (by Lemma~\ref{lem:BL}) and thus has its global minimum at $0$.
\end{proof}

The threshold $b^2<1/2$ will appear through the following integrability estimate for the covariance.
It is completely analogous to its Euclidean version.

\begin{lemma} \label{lem:C-bd}
  For $x,y\in \R^2$ such that $d_H(x,y)\leq \ell = L^j$,
  \begin{equation}
    C_j(x,y)
    = \log \frac{\ell}{d_H(x,y)} + O(1)
    ,
  \end{equation}
  and $C_j(x,y)=0$ if $x,y$ are not in the same $j$-block $B$.
  In particular, for any $a < 2$ and $b\in \R$,
  \begin{equation}
    \int_B C_j(x,y)^b e^{a C_j(x,y)} \, dy \lesssim |B| = \ell^2.
  \end{equation}
\end{lemma}

\begin{proof}
The proof is a direct computation using \eqref{e:C-decomp} and $\gamma_k Q_k(x,y)= (\log L) \1_{d_H(x,y) \leq L^k}$.
By hierarchical translation invariance we may assume that $B$ is the $j$-block containing $0\in \R^2$.
Then,
\begin{equation}
  \int_{d_H(0,x)\leq \ell} \pbb{\frac{\ell}{d_H(0,x)}}^{a} \pbb{\log \frac{\ell}{d_H(0,x)}}^b \, dx
  =
  \ell^2 \int_{d_H(0,x)\leq 1} d_H(0,x)^{-a} (\log d_H(0,x)^{-1})^{b} \, dx
  \lesssim \ell^2,
\end{equation}
by convergence of the integral if $0\leq a<2$. To see convergence of the integral, observe that if
$s=d_H(x,y)$, then $x$ and $y$ lie in the same $\frac{\log s}{\log L}$-block. So the Euclidean distance between $x$ and $y$ 
is at most the diameter of such a block, which is $\sqrt{2}d_H(x,y)$. The integral can therefore be bounded above by the classical integral of inverse-powers of $\abs{x}$.
\end{proof}

\begin{proof}[Proof of Proposition~\ref{prop:H0}]
By the hierarchical structure, we can restrict to a single $j$-block $B$.

In the remainder of the proof, we estimate the right-hand side of \eqref{e:H-BL} for $b^2 < 1/2$.
Using that $X/(1+X) \geq X-X^2$ as quadratic forms for $X\geq 0$ the following bound is convenient for $b^2<1/2$:
\begin{equation}
  C_j^{1/2} H_j C_j^{1/2}
  \geq \PP_{-N-1,j}\qa{ C_j^{1/2} H_{-N-1} C_j^{1/2}} - \PP_{-N-1,j}\qa{ (C_j^{1/2}H_{-N-1}C_j^{1/2})^2}
  ,
\end{equation}
which gives by projecting with $Q_j$ and $Q_jC_j = (\sum_{k\leq j}\gamma_k) Q_j$ that
\begin{equation}
  Q_{j}H_j(\varphi)Q_j
  \geq \PP_{-N-1,j}\qa{Q_j H_{-N-1} Q_j} - \PP_{-N-1,j}\qa{Q_j H_{-N-1} C_j H_{-N-1} Q_j}.
\end{equation}

To estimate the first term on the right-hand side from below, recall that
\begin{align}
  H_{-N-1}(\varphi) = \diag(H_{xx}(\varphi)), \qquad H_{xx}(\varphi) = \mu_{-N-1}\cosh(2b\varphi(x)),
\end{align}
where $\mu_{-N-1} = L^{-2b^2(N+1)} \mu$.
By Girsanov's formula, with $\mu_j = L^{2b^2 j}\mu$,
\begin{align}
  \Eg_{C_{j}}[\mu_{-N-1} e^{\pm 2b\zeta(x)}F(\zeta)]
  &= \mu_{-N-1} e^{+2b^2 C_j(x,x)} \Eg_{C_{j}}[F(\zeta\pm 2 b C_j(x,\cdot))]
  \nnb
  &= \mu_j \Eg_{C_{j}}[F(\zeta\pm 2 b C_j(x,\cdot))],
\end{align}
and we obtain
\begin{align}
  \PP_{-N-1,j}\qa{ H_{xx}} (\varphi)
  &=
  \PP_{-N-1,j}\qa{ \mu_{-N-1} \cosh(2b(\varphi(x)+\zeta(x))) }(\varphi)
  \nnb
  &= \frac12 e^{2b\varphi(x)} \PP_{-N-1,j}\qa{ \mu_{-N-1} e^{2b\zeta(x)}}(\varphi) + \frac12 e^{-2b\varphi(x)} \PP_{-N-1,j}\qa{ \mu_{-N-1} e^{-2b\zeta(x)}}(\varphi)
    \nnb
  &= \frac12 \mu_j e^{2b\varphi(x)} \frac{\Eg_{C_j}[e^{-V_{-N-1}(\varphi+\zeta + 2b C_j(x,\cdot))}]}{\Eg_{C_j}[e^{-V_{-N-1}(\varphi+\zeta)}]}
    + \frac12 \mu_j e^{-2b\varphi(x)} \frac{\Eg_{C_j}[e^{-V_{-N-1}(\varphi+\zeta - 2b C_j(x,\cdot))}]}{\Eg_{C_j}[e^{-V_{-N-1}(\varphi+\zeta)}]} 
    \nnb
  &\geq \mu_j \cosh(2b\varphi(x)) \min_{\pm} e^{-V_j(\varphi\pm 2b C_j(x,\cdot))+V_j(\varphi)}
    .
\end{align}
The term involving the renormalized potential is estimated using the Jensen bound \eqref{e:VB-lbub} if $\varphi=0$
using Lemma~\ref{lem:C-bd}:
\begin{equation} \label{e:Vj-Jensen-L2}
  -V_j(2bC_j(x,\cdot))+V_j(0)
  \geq -O(\mu \ell^{2b^2}) \int_B \cosh(4b^2 C_j(x,y)) \, dy
  \gtrsim -\mu\ell^{2+2b^2}
\end{equation}
for $b^2 < 1/2$. Thus
\begin{equation}
  e^{-V_j(2bC_j(x,\cdot))+V_j(0)} \geq 1-O(\mu\ell^{2+2b^2}).
\end{equation}
We conclude that
\begin{align}
  \PP_{-N-1,j} [H_{xx}] (0) \geq \mu \ell^{2b^2} (1-O(\mu \ell^{2+2b^2})).
\end{align}

Similarly, for the second order term we have the following upper bound from Girsanov and the Jensen bound \eqref{e:VB-lbub}:
\begin{align}
  \PP_{-N-1,j}[H_{xx}H_{yy}](0)
  &  \lesssim (\mu \ell^{2b^2})^2 \max_{\pm} e^{\pm 4b^2 C_j(x,y)}   
    \max_{\pm} e^{-V_j(2bC_j(x,\cdot)\pm 2bC_j(y,\cdot))+V_j(0)}
    \nnb
    &\lesssim (\mu \ell^{2b^2})^2 \max_{\pm} e^{\pm 4b^2 C_j(x,y)}  
    e^{V_j(0)} 
  \lesssim (\mu \ell^{2b^2})^2 \max_{\pm} e^{\pm 4b^2 C_j(x,y)}  
\end{align}
where we used 
\begin{equation}
  \mu_{-N-1}^2 \Eg_{C_j}[e^{2b\zeta(x) \pm 2b\zeta(y)}] = \mu_j^2 e^{\pm 4b^2 C_j(x,y)} \approx \mu_j^2 \qa{\frac{\ell}{d_H(x,y)}}^{\pm 4b^2}.
\end{equation}
Thus, again using that $b^2<1/2$, 
by Lemma~\ref{lem:C-bd}:
\begin{equation}
  \PP_{-N-1,j}[Q_jH C_j H Q_j]
  \lesssim (\mu \ell^{2b^2})^2 \int_{d_H(0,x)\lesssim \ell} \max_{\pm} e^{\pm 4b^2 C_j(x,y)} C_j(x,y) \, dx
  \lesssim (\mu \ell^{2b^2})^2 \ell^2
  .
\end{equation}
In summary,
\begin{equation}
  \ell^2 H_j(0) \geq \ell^2 \PP_{-N-1,j} [Q_jH Q_j]   - \ell^2 \PP_{-N-1,j}[Q_jH C_j H Q_j] \geq \mu\ell^{2+2b^2} - O(\mu \ell^{2+2b^2})^2
\end{equation}
which is the claimed bound.
\end{proof}

\subsection{Convexity for $b^2 \in [1/2,1)$: Proof of Proposition~\ref{prop:H0-L1}}

We start from the same application of the Brascamp--Lieb inequality \eqref{e:H-BL}
now followed by the quadratic form inequality $X/(1+X) \geq X-C_pX^p$ for $X\geq 0$ with $p\in (1,2)$ to be chosen:
\begin{equation}
  C_j^{1/2} H_j C_j^{1/2}
  \geq \PP_{-N-1,j}\qa{ C_j^{1/2} H_{-N-1} C_j^{1/2}} - C_p\PP_{-N-1,j}\qa{ (C_j^{1/2}H_{-N-1}C_j^{1/2})^p}
  .
\end{equation}
To bound $\PP_{-N-1,j}[\hat H^p]$ where $\hat H=C_j^{1/2} H_{-N-1} C_j^{1/2}$, we use the following interpolation inequality.

\begin{lemma}
Let $p\in (1,2)$. Then for any symmetric positive-definite matrix $\hat H$ and any vector $f$,
\begin{equation}
  (f,\hat H^p f) \leq (f,\hat H f)^{2-p} (f,\hat H^2 f)^{p-1}.
\end{equation}
\end{lemma}

\begin{proof}
Let $\lambda_i \geq 0$ and $e_i$ be the eigenvalues and eigenvectors of $\hat H$.
Using that
\begin{equation}
  (\lambda^{p-\alpha})^{1/(2-p)}
  = \lambda
  , \qquad
  (\lambda^\alpha)^{1/(p-1)} = \lambda^2,
\end{equation}
with $\alpha=2(p-1) \in (0,p)$, then
\begin{equation}
  \sum_i \lambda_i^p (f,e_i)^2
  = \sum_i \lambda_i^{p-\alpha} \lambda_i^{\alpha} (f,e_i)^2
  \leq \qa{ \sum_i \lambda_i (f,e_i)^2}^{2-p}\qa{ \sum_i \lambda_i^{2} (f,e_i)^2}^{p-1}
\end{equation}
by H\"older's inequality with conjugate exponents $1/(2-p)$ and $1/(p-1)$
which are both in $(1,\infty)$ for $p\in (1,2)$.
\end{proof}

Therefore, now by H\"older in the probability distribution,
\begin{align}
  \PP_{-N-1,j}\qa{ (f,\hat H^p f)}
  &\leq \PP_{-N-1,j}\qa {(f,\hat H f)^{2-p} (f, \hat H^2 f)^{p-1} }
  \nnb
  &\leq \PP_{-N-1,j}\qa {(f,\hat H f)^{q(2-p)}}^{1/q} \PP_{-N-1,j}\qa{(f, \hat H^2 f)^{q'(p-1)} }^{1/q'}
\end{align}
where $q>1$ and $1/q'=1-1/q = (q-1)/q$ so $q'=q/(q-1)$.
Let $1_j$ be the $L^2$ normalized indicator function on a fixed block $B$ of scale $j$.
Since $C_j^{1/2} 1_j = \gamma_{-N,j}^{1/2} 1_j$ with $\gamma_{-N,j} \approx L^{2j}$ (see Appendix~\ref{app:hier})
we thus obtain
\begin{align} \label{e:He-L1-BLp}
  L^{2j} (1_j, H_j 1_j)
  &\geq L^{2j}\PP_{-N-1,j}\qa{ (1_j, H_{-N-1} 1_j)}
  \nnb
  &\qquad
  - C_p L^{2pj} \PP_{-N-1,j}\qa {(1_j, H_{-N-1} 1_j)^{q(2-p)}}^{1/q}
  \nnb
  &\qquad
  \qquad \times
  \PP_{-N-1,j}\qa{(1_j, H_{-N-1} C_j H_{-N-1} 1_j)^{q'(p-1)} }^{1/q'}
  .
\end{align}

We need to bound the first expectation on the right-hand side from below and the two expectations in the second term from above.
The proof uses the following $L^p$ bounds for the Gaussian multiplicative chaos (GMC).
Let $M^\pm$ be the following version of the GMC defined in terms of $\pm \varphi$ (and normalized with explicit
power of $\epsilon$):
\begin{equation} \label{e:M-def}
  M(B) = \frac12(M^+(B)+M^-(B)), \qquad
  M^\pm(B)
  =
  \int_B \epsilon^{2b^2} e^{\pm 2b \varphi(x)} \, dx.
\end{equation}
Let $p_*(b) >1$ be such that $\xi(p_*(b))=2$ where
\begin{equation}
  \xi(p) = (2+2b^2)p - 2b^2 p^2.
\end{equation}

\begin{lemma} \label{lem:GMC-Lp}
  For all $p \in (0,p_*(b))$ and all blocks $B$ at scale $k\leq j$:
  \begin{equation} \label{e:GMC-moment}
    \Eg_{C_j}[M^\pm(B)^p] \lesssim L^{(2+2b^2)pj} L^{(k-j)\xi(p)},
  \end{equation}
  with implicit constant depending on $p$ and $b$.
\end{lemma}

This lemma is essentially proved in \cite[Proposition~3.7]{MR2642887}.
In their paper, $\lambda = 2b$, $R=L^{j}$, $d=2$, and $\xi(p)$ is denoted $\zeta(p)$.
That paper however is for the Euclidean rather than the hierarchical setting,
and the simpler hierarchical argument can in fact be found in slightly different notation in \cite{MR431355,MR3497718}.
In all of these references, we note the important notational difference that the Wick ordering is with respect to scale $R=L^j$
rather than unit scale.
Because of these differences, we provide a proof.
  
\begin{proof} 
By translation invariance, we may and will assume that $B=[0,L^k)^2$. Under $C_j$, we can write, using the hierarchical structure, that 
\begin{equation}
(\varphi(x))_{x\in B}\stackrel{(d)}{=}(G_{j,k}+\psi_j(x L^{-k}))_{x\in B}
\end{equation}
where $\psi$ is a hierarchical GFF with covariance
$C_{-N-k,0}$  and $G_{j,k}$ is an independent centered Gaussian of variance $ (j-k)\log L $. 
It follows that 
\begin{align}
\Eg_{C_{-N,j}}[M^\pm(B)^p]
&= \E[e^{\pm 2b G_{j,k}p} ]L^{(2+2b^2)pk} \Eg_{C_{-N-k,0}} [M^{\pm}([0,1]^2)^p]
\nnb
&= L^{2b^2p^2(j-k)} L^{(2+2b^2)pk} \Eg_{C_{-N-k,0}} [M^{\pm}([0,1]^2)^p]
\nnb
&= L^{(k-j)\xi(p)} L^{(2+2b^2)pj} \Eg_{C_{-N-k,0}} [M^{\pm}([0,1]^2)^p]
.
\end{align}
The conclusion follows from the fact that, for $p\in (0,p_*(b))$, $\Eg_{C_{-N,0}} [M^{\pm}([0,1]^2)^p]$ is bounded uniformly in $N$, by \cite[Theorem 2]{MR431355}.
\end{proof}

\begin{lemma} \label{lem:H0-L1-first}
  For $p \in (1,p_*(b))$,
  \begin{equation}
    L^{2j} \PP_{-N-1,j}\qa{ (1_j, H_{-N-1} 1_j)}(0) \geq \mu L^{(2+2b^2)j} - O(\mu L^{(2+2b^2)j})^p.
  \end{equation}
\end{lemma}

\begin{proof}
  Let $X=\frac12 (M^+(B)+M^-(B))$ where $B$ is a block of scale $j$. Then,
  recalling that $1_j$ is the normalized indicator function on $B$,
  \begin{equation} \label{e:PNj-lb}
    L^{2j}\PP_{-N-1,j} \qb{ (1_j,H_{-N-1} 1_j)} (0)
    = \frac{\Eg_{C_j}[\mu X e^{-\frac{\mu}{4b^2}X}]}{\Eg_{C_j}[e^{-\frac{\mu}{4b^2}X}]}
    \geq \Eg_{C_j}[\mu Xe^{-\frac{\mu}{4b^2}X}]
    .
  \end{equation}
  Using that $Xe^{-X} \geq X-C_pX^p$ for $X \geq 0$ we get
  \begin{equation}
    L^{2j}\PP_{-N-1,j}\qb{ (1_j,H_{-N-1} 1_j)}(0) \geq \mu \Eg_{C_j}[X] - C_p\mu^p \Eg_{C_j}[X^p] = \mu L^{(2+2b^2)j} - C_p\mu^p \Eg_{C_j}[X^p].
  \end{equation}
  For $p \in (0,p_*(b))$ the second term on the right-hand side is bounded as needed by Lemma~\ref{lem:GMC-Lp}.
\end{proof}

\begin{lemma} \label{lem:H0-L1-second}
For $s \in (0,p_*(b))$,
\begin{equation} \label{e:H0-L1-second-1}
  \PP_{-N-1,j}\qa {L^{2sj} (1_j, H_{-N-1} 1_j)^{s}}(0) \lesssim (L^{(2+2b^2)j}\mu)^s
\end{equation}
and, if additionally $s>1$, then
\begin{equation}\label{e:H0-L1-second-2}
  \PP_{-N-1,j}\qa{L^{2sj}  (1_j, H_{-N-1}C_j H_{-N-1} 1_j)^{s/2} }(0) \lesssim (L^{(2+2b^2)j}\mu)^{s}
  .
\end{equation}
\end{lemma}

\begin{proof}
  We recall that $e^{+V_j(0)} = 1+O(L^{(2+2b^2)j}\mu)$ and that $V_{-N-1}(\zeta) \geq 0$. Therefore
  for any nonnegative function $F \geq 0$ we have the following simple upper bound
  in terms of the Gaussian expectation:
  \begin{align}
    \PP_{-N-1,j}[F](0)
    = e^{V_j(0)}\Eg_{C_j}\qB{F(\zeta) e^{-V_{-N-1}(\zeta)}} \leq
    (1+O(\mu L^{(2+2b^2)j})) \Eg_{C_j}\q{F}.
  \end{align}
  A slightly sharper estimate without the $1+O(\mu L^{(2+2b^2)j})$ prefactor follows for quasi-concave $F$ (as in our application)
  from the Gaussian correlation inequality, but it is not necessary for us.

  The first estimate \eqref{e:H0-L1-second-1} follows directly from the $L^p$ bound on the GMC stated in Lemma~\ref{lem:GMC-Lp}.
  For the second estimate \eqref{e:H0-L1-second-2}, we start with
  \begin{align}
    &L^{2js} \Eg_{C_j}\qa{(1_j, H_{-N-1}(\zeta) C_j H_{-N-1}(\zeta) 1_j)^{s/2}}
      \nnb
    &\lesssim
    \mu^s\Eg_{C_j}\qa{\pa{\int_{B\times B} C_j(x,y) \epsilon^{4b^2}\cosh(2b\zeta(x))\cosh(2b\zeta(y)) \, dx\,dy}^{s/2}}
    ,
  \end{align}
  Decomposing $C_j(x,y) = \sum_{k=-N}^j \gamma_k Q_k(x,y)=  (\log L) \sum_{k=-N}^j \1_{j(x,y) \leq k}$ the right-hand of the last display side equals
  \begin{align}
   (\log L)^{s/2} \mu^s\Eg_{C_j}&\qa{\pa{\sum_{k={-N}}^j \int_{B\times B} L^{2k}Q_k(x,y) \epsilon^{4b^2} \cosh(2b\zeta(x))\cosh(2b\zeta(y)) \,dx\,dy}^{s/2}}
       \nnb
	  &=  (\log L)^{s/2} \mu^s\Eg_{C_j}\qa{\pa{\sum_{k={-N}}^j \sum_{b\in \cB_k(B)}\int_{b\times b} \epsilon^{4b^2}  \cosh(2b\zeta(x))\cosh(2b\zeta(y)) \,dx\,dy}^{s/2}}
      \nnb
     &= (\log L)^{s/2} \mu^s\Eg_{C_j}\qa{\pa{\sum_{k={-N}}^j \sum_{b\in \cB_k(B)}\pa{\int_{b} \epsilon^{2b^2}\cosh(2b\zeta(x)) \, dx}^2}^{s/2}}
      \nnb
    &= (\log L)^{s/2} \mu^s\Eg_{C_j}\qa{\pa{\sum_{k={-N}}^j \sum_{b\in \cB_k(B)}M(b)^2}^{s/2}},
  \end{align}
  where the symmetric GMC $M$ is defined in \eqref{e:M-def}.
  Using subadditivity of $f(x)=x^{s/2}$ to move the power $s/2<1$ into the sum, this can be bounded by
  \begin{equation}
    (\log L)^{s/2} \mu^s\sum_{k={-N}}^j \sum_{b\in \cB_k(B)} \Eg_{C_j}\qa{M(b)^s}
    =
    (\log L)^{s/2} \mu^s\sum_{k={-N}}^j L^{2(j-k)} \max_{b\in \cB_k}\Eg_{C_j}\qa{M(b)^s}.
  \end{equation}
  Applying \eqref{e:GMC-moment} to bound the GMC moment $\Eg_{C_j}\qa{M(b)^s}$ where $|b|=L^{2k}$, the last display is equal to
  \begin{align}
    &(\log L)^{s/2} \mu^s L^{(2+2b^2)sj}\sum_{k={-N}}^j L^{2(j-k)} L^{(k-j)\xi(s)}
    \nnb
    &=  (\log L)^{s/2} \mu^s L^{(2+2b^2)sj} \sum_{k={-N}}^j  L^{-(\xi(s)-2)(j-k)}
    \lesssim \mu^s L^{(2+2b^2)sj}
  \end{align}
  where the last inequality follows from $\xi(s) > 2$.
\end{proof}

\begin{proof}[Proof of Proposition~\ref{prop:H0-L1}]
We start from \eqref{e:He-L1-BLp} and need the exponents to satisfy
\begin{equation} \label{e:pq-cond}
  q(2-p) < p_*(b), \qquad 2q(p-1)/(q-1) = 2q'(p-1) < p_*(b).
\end{equation}
These conditions are equivalent to
\begin{equation}
  2-\frac{p_*(b)}{q} < p < 1+\frac{q-1}{2} \frac{p_*(b)}{q}
\end{equation}
and there is a valid choice of $p$ if and only if
\begin{equation}
  1< q < \frac{p_*(b)}{2-p_*(b)}.
\end{equation}
Using the $q$ chosen in this way the condition \eqref{e:pq-cond}
is satisfied for any $p \in (2-\frac{p_*(b)}{q}, 1+\frac{q-1}{2}\frac{p_*(b)}{q}) \cap (1,2)$.
Thus with $p$ and $q$ chosen we can apply
Lemma~\ref{lem:H0-L1-first} for the first term in \eqref{e:He-L1-BLp} and
Lemma~\ref{lem:H0-L1-second} for the second term to get
\begin{equation}
  L^{2j}(1_j,H_j1_j) \geq L^{(2+2b^2)j} \mu  - C_p (L^{(2+2b^2)j}\mu)^{p}
\end{equation}
as claimed.
\end{proof}

\subsection{Convexity for large scales: Proof of Corollary~\ref{cor:Hej0}}

\begin{proof}[Proof of Corollary~\ref{cor:Hej0}]
We again start from the Brascamp--Lieb inequality, except that we now use \eqref{e:BL-large} from scale $j-1$ to scale $j$ rather than from scale $-N$ to $j$:
\begin{equation} \label{e:BL-large-bis}
  \gamma_j Q_jH_j(\varphi)Q_j
  \geq
  \E_{\mu_{j-1,j}^\varphi}\qa{\frac{\gamma_jQ_j H_{j-1} Q_j}{1+\gamma_j Q_j H_{j-1}Q_j}}
  .
\end{equation}
The claim then follows by induction in $j$.
Indeed, we assume that $\gamma_{j} Q_j H_j Q_j \geq { \rho_j} Q_j$ for some constants $\rho_j>0$.
Then in particular $\gamma_{j+1} Q_{j+1} H_j Q_{j+1} \geq L^2\rho_j Q_{j+1}$
where we used that $Q_{j+1}Q_j =Q_{j+1}$ and that $\gamma_{j+1} = L^2\gamma_j$.
Therefore 
\begin{equation}
  \gamma_{j+1} Q_{j+1}H_{j+1}(\varphi)Q_{j+1} \geq
  \E_{\mu_{j,j+1}^\varphi}\qa{ \frac{\gamma_{j+1} Q_{j+1} H_j Q_{j+1}}{1+\gamma_{j+1} Q_{j+1} H_j Q_{j+1}}}
  \geq \frac{L^2\rho_j}{1+L^2 \rho_j} Q_{j+1}
  ,
\end{equation}
i.e., the inductive assumption is advanced with
\begin{equation}
  \rho_{j+1} \geq \frac{L^2\rho_j}{1+L^2 \rho_j}.
\end{equation}
The claim follows by induction starting from $j=j_0(b,\mu)$ which is such that $\rho_j \geq c>0$
for a constant $c$ independent of $\mu$ and $N$.
Since $L^2c/(1+L^2c) \geq c$ for $c>0$ small enough, we then obtain $\rho_{j+1} \geq c$.
In fact, asymptotically, $\rho_j \sim 1-L^{-2}$.
\end{proof}

\subsection{Cram\'er--Rao upper bound: Proof of Proposition~\ref{prop:CR}}

\begin{proof}[Proof of Proposition~\ref{prop:CR}]
  The Hessian of the renormalized potential can also be represented in terms of the covariance matrix of the fluctuation measure $\mu_{-N-1,j}^\varphi$
  defined in \eqref{e:Pkj-mukj-def} as
  \begin{equation}
    \He V_j(\varphi) = C_j^{-1} - C_j^{-1} {\rm Cov}(\mu_{-N-1,j}^\varphi) C_j^{-1}.
  \end{equation}
  Again using the definition of $\mu_{-N-1,j}^\varphi$,
  the second term is bounded by the Cram\'er--Rao inequality
  (see, for example, \cite[Appendix~A]{MR4683324}):
  \begin{equation}
    {\rm Cov}(\mu_{-N-1,j}^\varphi) \geq \pB{C_j^{-1} + \E_{\mu_{-N-1,j}^\varphi}[H_{-N-1}(\varphi+\zeta)]}^{-1}.
  \end{equation}
  Since $V_j$ is additive over $j$-blocks, it is sufficient to consider $\varphi$ constant, and
  we do this from now on.
  Using the same notation as in \eqref{e:PNj-lb} (but not necessarily requiring $\varphi=0$),
  in particular that $1_j$ is the normalized indicator function on a $j$-block $B$ and that $X(\zeta) = X_\varphi(\zeta) = \int_B L^{-2b^2(N+1)}\cosh(2b\varphi+2b\zeta)$,
  \begin{equation}
    \E_{\mu_{-N-1,j}^\varphi}[H_{-N-1}(\varphi+\zeta)] 
    =
    \PP_{-N-1,j} \qb{ (1_j,H_{-N-1}(\varphi+\zeta) 1_j)} (\varphi)
    = L^{-2j} \frac{\Eg_{C_j}[\mu X e^{-\frac{\mu}{4b^2}X}]}{\Eg_{C_j}[e^{-\frac{\mu}{4b^2}X}]}
    ,
  \end{equation}
  where the first equality follows from the fact that the matrix on the left-hand side is proportional to the 
  identity (for constant $\varphi$) and the second equality uses that $L^{-2j}\mu X_\varphi(\zeta) = (1_j, H_{-N-1}(\varphi+\zeta) 1_j)$ as random variables.
  Using the one-dimensional FKG inequality, which  holds for all real-valued random variables $X$,
  using that the function $x$ is increasing and that $e^{-x}$ is decreasing, thus
  \begin{equation}
    \E_{\mu_{-N-1,j}^\varphi}[H_{-N-1}(\varphi+\zeta)] 
    \leq L^{-2j} \Eg_{C_j}[\mu X]
    = L^{2b^2j}\mu\cosh(2b\varphi),
  \end{equation}
  where the final equality is a Gaussian calculation analogous to that in \eqref{e:VN-expect} and uses that $|B| = L^{2j}$.
  Thus, with $\mu_j = L^{2b^2j} \mu$ for $\varphi$ constant, we have 
  \begin{equation}
    {\rm Cov}(\mu_{-N-1,j}^\varphi) \geq \pB{C_j^{-1} + \mu_j \cosh(2b\varphi) \id}^{-1}
  \end{equation}
  and
  \begin{equation}
    C_j^{1/2}\He V_j(\varphi)C_j^{1/2}
    \leq 1 - (1+\mu_j \cosh(2b\varphi) C_j)^{-1}
    = \frac{\mu_j \cosh(2b\varphi) C_j}{1+\mu_j \cosh(2b\varphi) C_j}.
  \end{equation}
Projecting with $Q_j$ and using that with $\gamma_{-N,j} = \sum_{k\leq j}\gamma_k$ and
thus $\sum_{k\leq j}\gamma_k^{1/2} \geq \gamma_{-N,j}^{1/2}$,
\begin{align}
  Q_jC_j &= (\sum_{k\leq j} \gamma_k) Q_j = \gamma_{-N,j} Q_j
  \\
  Q_jC_j^{1/2} &= (\sum_{k\leq j} \gamma_k^{1/2}) Q_j \geq \gamma_{-N,j}^{1/2} Q_j
\end{align}
we obtain
\begin{equation}
  \gamma_{-N,j} Q_j \He V_j(\varphi) Q_j \leq \frac{\mu_j \cosh(2b\varphi) \gamma_{-N,j}}{1+ \mu_j \cosh(2b\varphi)\gamma_{-N,j}} Q_j.
\end{equation}  
Since $\gamma_{-N,j} \geq \gamma_j = (\log L) L^{2j}$ and $\mu_j = L^{2b^2 j} \mu$ the claim follows.
\end{proof}

\section{Proof of main results}

To prove the main results, we first prove the log-Sobolev inequality and mass gap
for the regularized measure \eqref{e:nu-NM}, uniformly in the ultraviolet and infrared cutoffs.
Either implies the needed tightness in order to take the limit along subsequences.
The existence of the actual limit is a result of monotonicity established subsequently.

\subsection{Log-Sobolev inequality}


To prove (the at first regularized version of) the log-Sobolev inequality, we apply the multiscale
Bakry--\'Emery criterion of \cite{MR4303014} and more precisely the special instance for the
hierarchical setting discussed in \cite[Example~2.7]{MR4303014}.
This  application relies on the Hessian lower bounds of the renormalized potential established  in
Propositions~\ref{prop:H0}--\ref{prop:H0-L1} and Corollary~\ref{cor:Hej0}. To be precise, it
uses a slight generalization in which the covariances are continuously interpolated. This generalization
follows from exactly the same proofs (and alternatively can be derived from the stated estimates
by an argument completely analogous to that in the proof of Corollary~\ref{cor:Hej0}).

\begin{proof}[Proof of regularized version of \eqref{e:thm-LS}]
Define
\begin{equation}
  \dot C_t = \sum_j 1_{(j-1,j]}(t) L^{2j} Q_j, \qquad C_t = \int_{-N-1}^t \dot C_s \, ds,
\end{equation}
so that for $t\in \Z$ (always with $t \geq -N$), the definition of $C_t$ coincides with \eqref{e:C-decomp}.
For $t\in \R$, the renormalized potential $V_t$ is defined exactly as in \eqref{e:Vj}.
The estimates from Section~\ref{sec:renpot} are straightforward to extend from $j\in\Z$ to $t\in \R$.
In particular, the extensions of the estimates of Propositions~\ref{prop:H0}--\ref{prop:H0-L1} and Corollary~\ref{cor:Hej0} show that,
uniformly in $\varphi \in X_t = X_{\lceil t \rceil}$,
\begin{equation}
  \dot C_t^{1/2} \He V_t(\varphi) \dot C_t^{1/2} \gtrsim \min\{\mu L^{(2+2b^2)t},1\}.
\end{equation}
By \cite[Example~2.7]{MR4303014}, the log-Sobolev constant $\gamma$ of the measure
$\nu^{\hShG(b,\mu|-N,M)}$ is bounded by
\begin{align}
  \frac{1}{\gamma}
  &\lesssim \int_{-N-1}^{\infty} \exp\pa{-c\int_{-N-1}^t \min\{\mu L^{s(2+2b^2)},1\}  \,ds} L^{2t}\,dt\nnb
  &\lesssim   \int_{-N-1}^{\infty} \exp\pa{-c\mu L^{(t\wedge t_0)(2+2b^2)} -c(t-t_0)_+ } L^{2t} \, dt
   \approx L^{2j_0} \approx \mu^{-1/(1+b^2)},
\end{align}
where we used that $t_0=j_0$ is defined so that $L^{(2+2b^2)j_0}\mu \approx 1$. In summary,
for any cylindrical test function with compact support,
\begin{equation}
  \ent_{\nu_{\hShG(b,\mu|-N,M)}} (F) \lesssim \mu^{-1/(1+b^2)} \E_{\nu^{\hShG(b,\mu|-N,M)}}\qa{\|\nabla \sqrt{F}\|_{L^2}^2}.
\end{equation}
This is the regularized version of \eqref{e:thm-LS} formulated on the lattice. 
Upon passing to the continuum form of the inequality, the inequality is preserved if lattice functions are replaced with their piecwise-constant analogues in the continuum.
This gives the statement for all hierarchical cylinder functions uniformly in the cutoffs $-N,M$. The result follows because $\phi_{-N,M}(f)$
converges in distribution, so if $F$ is regulated by replacing it with $F+\eps$, the inequality passes to the limit, after which the mollifier may be removed.
\end{proof}

\subsection{Massive correlation decay}

To explain the idea of the argument, first consider the sinh-Gordon model on the unit lattice $\Lambda \subset \Z^d$:
\begin{equation}
  \avg{F} \propto \int_{\R^\Lambda} e^{\frac12 (\varphi,\Delta\varphi) - \frac{\mu}{4b^2}\sum_{x\in\Lambda} \cosh(2b \varphi_x)} F(\varphi) \, d\varphi,
\end{equation}
where $\Delta$ is the lattice Laplacian on $\Lambda$.
Then the massive decay is straightforward from the strict convexity of $\cosh$.
The simplest proof uses the Brydges--Fr\"ohlich--Spencer (BFS) random walk representation \cite{MR648362}.
Indeed, using this representation the two-point function is given by
\begin{align}
  \avg{\varphi(x)\varphi(y)}
  &= \int_0^\infty \E_{x}\qa{ 1_{X_T=y} \frac{Z(L_T)}{Z(0)}} \, dT
\end{align}
where $X$ is the random walk with generator $\Delta$, $L_T = (\int_0^T \1_{X_t=x})_{x\in\Lambda}$ its local time up to time $T$,
and for $\ell: \Lambda \to [0,\infty)$,
\begin{align}
  Z(\ell) = \int e^{\frac12 (\varphi,\Delta \varphi)} e^{-\frac{\mu}{4b^2} \sum_x \cosh(2b\sqrt{\varphi_x^2+2\ell_x})} \, d\varphi.
\end{align}
Using that $\cosh(\sqrt{t}) = \frac12 t + W(t)$ with $W(t)$ increasing, i.e., $\cosh(\sqrt{t+s}) \geq \cosh(\sqrt{t}) +  \frac12 s$
for $t,s\geq 0$,
\begin{equation}
  \frac{Z(\ell)}{Z(0)} \leq e^{-\mu \sum_x \ell_x} \quad\text{and}\quad
  \sum_x L_{T,x} = T,
\end{equation}
and we conclude the pointwise bound
\begin{equation} \label{e:BFS-bound-lattice}
  0 \leq
  \avg{\varphi(x)\varphi(y)}
  \leq \int_0^\infty \E_{x}\qa{ 1_{X_T=y} e^{-\mu T}}\, dT
  = (-\Delta+\mu)^{-1}(x,y).
\end{equation}

The only feature that is used is that $\Delta$ is the generator of a random walk
and that $\cosh(\varphi)-\frac12 \varphi^2$ is an increasing function of $\frac12 \varphi^2$.
In the hierarchical case, the renormalized measure, defined by
\begin{equation}
\label{eq:nuj}
\E_{\nu_{j+1}}[F] \propto \Eg_{C_{j+1,M}} \Big[F(\varphi) e^{-V_j(\varphi)} \Big],
\end{equation}
where the renormalized measure $V_j$ defined in \eqref{e:Vj} has the same structure as the
original measure: It has a local single site potential that is even and convex
and the renormalized Laplacian remains ferromagnetic as verified in Lemma~\ref{lem:hiergen} in Appendix~\ref{app:hier}.
In particular, this measure admits a random walk representation and satisfies the the FKG inequality.
It is related to the full measure by the following identity, which holds for any $j \geq -N$:
\begin{equation}
  \E_{\nu}[F] = \E_{\nu_{j+1}}\qa{\E_{\mu_j^\varphi}[F]},
\end{equation}
where $\varphi$ is the field distributed according to $\nu$.

The proof of the two-point function decay of the hierarchical continuum sinh-Gordon model now follows in two steps:
\begin{enumerate}
\item
The two-point function of the full measure is dominated by that of the renormalized measure.
\item
The two-point function of the renormalized measure satisfies the massive correlation decay.
\end{enumerate}
The first step is stated as the following lemma.
For $x\in \R^2$, let $B_x\in \cB_j$ be the $j$-block containing $x$.
The renormalized field $\varphi\sim \nu_j$ is constant on the blocks in $\cB_j$ and we will
also write $\varphi(x) = \varphi_{B_x}$ in this situation.
Also recall that the full field $\varphi \sim \nu$ is constant on the blocks of the ultraviolet scale $-N$.
The next lemma shows that, as a consequence of the FKG inequality,
the two-point function of the full measure is bounded by the
two-point function of the renormalized measure.

\begin{lemma} \label{lem:twopoint-renorm}
For any $x,y \in \R^2$ with $d_H(x,y) > L^{j}$ for some $j>-N$,
\begin{equation}
  \label{eq:hshg-2pt}
  \E_{\nu}[\varphi(x) \varphi(y)]
  \leq
  \E_{\nu_j}[\varphi(x) \varphi(y)],
\end{equation}
where $\nu$ denotes the regularized sinh-Gordon measure $\nu^{\hShG(b,\mu|-N,M)}$
and $\nu_j$ is defined in \eqref{eq:nuj}.
\end{lemma}

\begin{proof}
It suffices to show the claim with $j$ replaced by $j+1$.
The full (regularized) sinh-Gordon measure $\nu$ decomposes into renormalized measure $\nu_{j+1}$ and fluctuation measure $\mu_j^\varphi$
defined in \eqref{e:Pkj-mukj-def} according to
\begin{equation}
  \E_{\nu}[F(\varphi)]
  = \E_{\nu_{j+1}}\qa{\E_{\mu_j^\varphi}[F(\zeta)]}.
\end{equation} 
By the hierarchical structure,
conditioned on the block-spin field $\varphi \sim \nu_{j+1}$ at scale $j+1$, the fluctuation measure $\mu_j^\varphi$
is a product measure over fields on $j$-blocks.
Then, for $j$-blocks $B_x \neq B_y$,
\begin{equation}
  \E_{\nu}[\varphi_x\varphi_y]
  = \E_{\nu_{j+1}}\qa{\E_{\mu_j^\varphi}[\zeta_x\zeta_y]}
  = \E_{\nu_{j+1}}\qa{\E_{\mu_j^\varphi}[\zeta_x] \E_{\mu_j^\varphi}[\zeta_y]} .
\end{equation}
For the inner expectations, the external field $\varphi$ can be assumed to be constant everywhere
(as it is constant on the block $B_x$ and $\zeta_x$ only depends on $\varphi$ restricted to $B_x$).
For constant external field $\varphi$, the conditional mean $\E_{\mu_j^\varphi}[\zeta_x]$ is increasing in $\varphi$:
\begin{equation}
  \partial_\varphi \E_{\mu_j^\varphi}[\zeta_x]
  \propto 
  \cov_{\mu_j^\varphi}[\zeta_x; (\zeta,1_B)]
  \propto
  \cov_{\mu_j^\varphi}[(\zeta,1_B); (\zeta,1_B)]
  >0.
\end{equation}
On the other hand, the mean is related to the renormalized potential according to (see e.g.~\cite[(3.60)]{MR4798104}):
\begin{equation}
    C_j \nabla V_j(\varphi) = \varphi - \E_{\mu_j^\varphi}[\zeta],
\end{equation}
and the $\varphi$-derivative of the left-hand side is increasing by convexity of $V_j$.
Thus all terms in the last display are increasing in $\varphi$ and have expectation $0$ under $\nu_{j+1}$.
It follows that
\begin{align}
  \E_{\nu_{j+1}}\qa{\E_{\mu_j^\varphi}[\zeta_x] \E_{\mu_j^\varphi}[\zeta_y]}
  &=\E_{\nu_{j+1}}\qa{\E_{\mu_j^\varphi}[\zeta_x](\varphi_y - C_j\nabla V_j(\varphi,y))}
  \nnb
  &=\E_{\nu_{j+1}}\qa{\E_{\mu_j^\varphi}[\zeta_x]\varphi_y}
  - \E_{\nu_{j+1}}\qa{\E_{\mu_j^\varphi}[\zeta_x]C_j\nabla V_j(\varphi,y))}
  \nnb
  &\leq\E_{\nu_{j+1}}\qa{(\varphi_x-C_j\nabla V_j(\varphi,x))\varphi_y}
  \nnb
  &=\E_{\nu_{j+1}}\qa{\varphi_x\varphi_y}
  -\E_{\nu_{j+1}}\qa{C_j\nabla V_j(\varphi,x)\varphi_y}
  \leq \E_{\nu_{j+1}}\qa{\varphi_x\varphi_y},
\end{align}
where the two inequalities follow from the FKG inequality satisfied by $\nu_{j+1}$ and that the involved terms have mean $0$.
\end{proof}

\begin{lemma} \label{lem:twopoint-renorm2}
  The renormalized measure $\nu_j$ has massive correlation decay with squared mass $\mu_j$ given by
  the lower bound in Propositions~\ref{prop:H0}--\ref{prop:H0-L1} respectively Corollary~\ref{cor:Hej0}:
  \begin{equation}
  \avg{\varphi(x)\varphi(y)}_{\nu_j}
  \lesssim (-\Delta_H + \mu_{j})^{-1}(x,y).
  \end{equation}
\end{lemma}

\begin{proof}
The bound \eqref{e:BFS-bound-lattice} from the BFS representation can be applied to the renormalized problem and gives mass $\mu_j$.
In more detail, since $V_j$ is symmetric and convex, we can write
\begin{equation}
  V_j(\varphi) = \frac12 \mu_j \varphi^2 + W_j(\frac12 \varphi^2),
\end{equation}
where $\mu_j>0$ can be chosen consistent with Section~\ref{sec:renpot} and
the function $W_j(\ell)$ increasing:
\begin{equation}
  W_j(\frac12 \varphi^2+\ell) \geq W_j(\frac12 \varphi^2).
\end{equation}
Now we apply the random walk representation to the renormalized measure
$\nu_j$ at scale $j$, defined in \eqref{eq:nuj}.
By Lemma~\ref{lem:hiergen} in Appendix~\ref{app:hier}, the Gaussian part of this measure has action
\begin{equation}
  C_{j,M}^{-1} = -\Delta_{H,j,M}+1/\gamma_{j,M}.
\end{equation}
Denoting by $\E$ the expectation of the random walk $X$ on $\cB_j$ with generator $\Delta_{H,j,M}$, for $B,B' \in \cB_j$,
\begin{align}\label{e:BFS-hier}
  \avg{\varphi_B\varphi_{B'}}_{\nu_j}
  &= \int_0^\infty \E_{B}\qa{ 1_{X_T=B'} \frac{Z_j(L_T)}{Z_j(0)}} \, dT
    \nnb 
  &\leq \int_0^\infty \E_{B}\qa{ 1_{X_T=B'} e^{-(1/\gamma_{j,M}+\mu_{j-1}) T}}\, dT
    \nnb
  &= (-\Delta_{H,j,M}+{ 1/\gamma_{j,M}} + \mu_{j-1})^{-1}(B,B')
  \lesssim (-\Delta_H+\mu_{j-1})^{-1} (B,B'),
\end{align}
where $(-\Delta_{H,j,M}+m^2)^{-1}(B,B')$ is covariance of $\varphi_B$ and $\varphi_{B'}$ under the
Gaussian measure with covariance $(-\Delta_{H,j,M}+m^2)^{-1}$ with the convention that $\varphi_x = \varphi_{B_x}$
and using that,
for any $\ell: \cB_j \to [0,\infty)$,
\begin{equation}
  Z_j(\ell) = \int e^{\frac12(\varphi,\Delta_{H,j,M} \varphi)} e^{-(\mu_{j-1}+1/\gamma_{j,M}) \sum_B (\frac12 \varphi_B^2+\ell_B) - \sum_B W_{j-1}(\frac12 \varphi_B^2+\ell_B)} \, d\varphi
\end{equation}
satisfies
\begin{equation}
  \frac{Z_j(\ell)}{Z_j(0)} \leq e^{-(\mu_{j-1}+1/\gamma_{j,M}) \sum_B \ell_B}.
\end{equation}
In the last inequality in \eqref{e:BFS-hier} we used \eqref{e:DeltaHj-DeltaH}.
This implies (using $\mu_j \approx \mu_{j-1}$ to simplify notation)
\begin{equation}
\avg{\varphi(x)\varphi(y)}_{\nu_j}
= \avg{\varphi_{B_x}\varphi_{B_y}}_{\nu_j}
\lesssim (-\Delta_H + \mu_{j})^{-1}(B_x,B_y)
= (-\Delta_H + \mu_{j})^{-1}(x,y),
\end{equation}
where again $\varphi_{B_x} = \varphi(x)$.
\end{proof}

\begin{proof}[Proof of regularized version of \eqref{e:thm-twopoint}]
Using
Lemma~\ref{lem:twopoint-renorm} in the first inequality below and Lemma~\ref{lem:twopoint-renorm2} in the second,
for $j(x,y)>j$ and $B_x$ and $B_y$ denoting $j$-blocks containing $x$ respectively $y$,
%
\begin{align}
  \avg{\varphi(x)\varphi(y)}_{\nu} 
  &\leq \avg{\varphi_{B_x}\varphi_{B_y}}_{\nu_j}
  \nnb
  &\lesssim (-\Delta_H + \mu_{j})^{-1}(B_x,B_y)
  = (-\Delta_H + \mu_{j})^{-1}(x,y).
 \end{align}
Choosing $j=j_0(b,\mu)$ we get that for $j(x,y) > j_0(b,\mu)$,
\begin{equation}
  \avg{\varphi(x)\varphi(y)}_{\hShG(b,\mu|-N,M)} \lesssim (-\Delta_H+1)^{-1}(x,y).
\end{equation}
The above condition on $x,y$ is equivalent to $d_H(x,y) \geq \mu^{-1/(2+2b^2)}$.
This proves the claim if $\mu\approx 1$.

The regularized sinh-Gordon measure has the following scale covariance: for $s = L^\sigma \in L^\Z$ with the notation $R_s\varphi(x) = \varphi(x/s)$,
\begin{equation} \label{e:measure-rescale}
  \avg{F}_{\hShG(b,\mu | -N,M)} = \avg{R_sF}_{\hShG(b,\mu s^{2+2b^2} | -N+\sigma|M+\sigma)},
\end{equation}
whereas the massive free field satisfies
\begin{equation}
  \avg{F}_{\hGFF(m)} = \avg{R_sF}_{\hGFF(m s)}
  .
\end{equation}
Therefore, rescaling, we obtain that for general $\mu$ with $s = L^\sigma \approx \mu^{-1/(2+2b^2)}$:
\begin{align}
  \avg{\varphi(x)\varphi(y)}_{\hShG(b,\mu | -N,M)}
  &\approx \avg{\varphi(x/s)\varphi(y/s)}_{\hShG(b,1 | -N+\sigma|M+\sigma)}
    \nnb
  &\lesssim \avg{\varphi(x/s)\varphi(y/s)}_{\hGFF(1)}
  = \avg{\varphi(x)\varphi(y)}_{\hGFF(1/s)} .
\end{align}
The right-hand side is $(-\Delta_H+\mu_\eff)^{-1}(x,y)$ with $\mu_\eff \approx s^{-2} = \mu^{2/(2+2b^2)}$ as claimed.
\end{proof}

\begin{proof}[Proof of scale equivariance~\eqref{e:thm-scaling}]
We begin once again with the approximate scale-symmetry of the cut-off hierarchical GFF, 
which leads to the identity 
\begin{equation}
  \avg{R_sF}_{\hShG(b,\mu | -N,M)} = \avg{F}_{\hShG(b,\mu s^{-2-2b^2} | -N+\sigma|M+\sigma)},
\end{equation}
Now for $F:\mathcal{S}'\to \bR$ bounded and continuous in the strong dual topology,
\begin{align}
\avg{R_s F}_{\hShG(b,\mu)} = \lim_{M\to \infty}\lim_{N\to \infty}\avg{R_sF}_{\hShG(b,\mu | -N,M)}&=\lim_{M\to \infty}\lim_{N\to \infty}\avg{F}_{\hShG(b,\mu s^{-2-2b^2} | -N+\sigma|M+\sigma)}
\nnb
&=\ev{F}_{\hShG(b,\mu s^{-2-2b^2})},
\end{align}
provided the limit exists, which we show next.
\end{proof}

\subsection{Continuum and infinite-volume limit}

Finally, we prove that $\nu^{\hShG(b,\mu|-N,M)}$ converges weakly to a probability measure $\nu^{\hShG(b,\mu)}$ on $\cS'(\R^2)$ as $-N\to-\infty$ and then $M\to\infty$.
In the following, we work on the probability space defined in terms of independent Gaussians for all increments of the (two-sided)
branching random walk representating the hierarchical covariance of all scales,
so that the regularized measures can all be defined simultaneously.

\paragraph{Continuum limit}

The existence of the continuum limit follows from the convergence of the Gaussian multiplicative chaos (GMC).
Indeed, let 
\begin{equation}
  M^\pm_{\epsilon}(\Lambda) = \int_{\Lambda} \epsilon^{2b^2} e^{\pm 2b\varphi_\epsilon(x)} \, dx
\end{equation}
be the regularized Gaussian multiplicative chaos measure with $\epsilon=L^{-N-1}$,
i.e., $\varphi_\epsilon$ has Gaussian distribution with covariance $C_{-N,M}$.
For $M$ fixed, this setup is equivalent to a multiplicative cascade \cite{MR829798,MR3497718}, and 
it is well known that $M^\pm_{\epsilon}$ converges in probability as $\epsilon=L^{-N-1}\to 0$ weakly 
as a random positive measure to a limit measure $M^\pm$. Let
\begin{equation}
  V_{N,M}(\Lambda) = \frac{\mu}{8b^2}\pB{M_\epsilon^+(\Lambda)+M_\epsilon^{-}(\Lambda)} 
\end{equation}
and
\begin{equation}
  V_{M}(\Lambda) = \frac{\mu}{8b^2}\pB{M^+(\Lambda)+M^{-}(\Lambda)} .
\end{equation}
Then also $e^{-V_{N,M}} \to e^{-V_M}$ in $L^1$.
This implies that the finite-volume continuum hierarchical sinh-Gordon measure $\nu^{\hShG(b,\mu|-\infty,M)} = \lim_{N\to\infty} \nu^{\hShG(b,\mu|-N,M)}$ exists
with expectation value
\begin{equation}
  \label{e:infvol-Lambda}
  \E_{\nu^{\hShG(b,\mu|-\infty,M)}}[F] 
  =
  \frac{\E\qa{e^{-V_{M}(\Lambda)}F}}
  {\E\qa{e^{-V_{M}(\Lambda)}}}
  ,
\end{equation}
where $\Lambda$ is an arbitrarily large union of $M$-blocks containing the support of $F$,
as discussed below Definition~\ref{defn:nu-NM}. In particular, by standard methods, the hierarchical GFF extends to a Gaussian process indexed by $L^2(\bR^2)$
with values in $\bigcap_p L^p(\mu_{\hGFF})$.  Therefore the $L^1$-convergence of $e^{-V_{-N,M}}\to e^{-V_{-\infty, M}}$
implies the convergence of $\phi_{-N,M}(f)$ to $\phi_{-\infty,M}(f)$ in distribution for $f\in L^2$.

\paragraph{Infinite-volume limit}

For the infinite-volume limit $M\to\infty$, we can (for example) start from the fact that
the log-Sobolev inequality and the Herbst argument \cite[Proposition~5.4.1]{MR3155209} imply the following uniform Gaussian integrability estimate
for any $f\in C_c^\infty(\R^2)$:
\begin{equation} \label{e:infvol-moment}
  \E_{\nu^{\hShG(b,\mu|-N,M)}}[e^{(\varphi,f)}] \leq e^{\frac{1}{2\gamma}\|f\|_{L^2}^2}
  .
\end{equation}
This already implies tightness of the measures $\nu^{\hShG(b,\mu|-N,M)}$
and therefore only convergence without passing to a subsequence is left to show.
Indeed, tightness follows, for example, via equicontinuity of the characteristic functional implied by the above uniform bound
and Lemma~5.2 of \cite{biermé2017generalizedrandomfieldslevys}.

It is verified in Lemma~\ref{lem:hiergen} that the measure $\nu^{\hShG(b,\mu|-N,M)}$ is ferromagnetic.
By the first Griffiths inequality, all correlation functions are therefore nonnegative, and it therefore suffices
to show that $\E_{\nu^{\hShG(b,\mu|-\infty,M)}}[e^{(\varphi,f)}]$ converges for nonnegative test functions $f$.
The reduction to nonnegative test functions using positivity of the correlation functions
follows as in the proof of \cite[Theorem~11.2.1]{MR887102}.
Further background on Griffiths inequalities can also be found in \cite{MR887102}.

For nonnegative $f$, we now show that the left-hand side of \eqref{e:infvol-moment} is increasing in $M$
and thus converges.
This is the analogue of the monotonicity of the generating functional in \cite[Theorem~10.2.2]{MR887102}.
To see this, since $\Lambda$ is arbitrary, we may assume that $\Lambda$ is
a union of $(M+1)$-blocks containing the support of $f$. It then suffices to show that the expectation with $\Lambda$ fixed is increasing in $M$.
This follows from the second Griffiths inequality. Indeed, it is also verified in Lemma~\ref{lem:hiergen} that, pointwise,
\begin{equation}
  C_{-N,M+1}^{-1} \leq C_{-N,M}^{-1}.
\end{equation}
Thus, with the matrix $R = C_{-N,M}^{-1} - C_{-N,M+1}^{-1}$ with nonnegative entries, with $t=1$,
\begin{equation}
  \Eg_{C_{-N,M+1}}[F] \propto \Eg_{C_{-N,M}}\qa{e^{t(\varphi,R\varphi)} F},
\end{equation}
whereas, when $t=0$, trivially,
\begin{equation}
  \Eg_{C_{-N,M}}[F] = \Eg_{C_{-N,M}}\qa{e^{t(\varphi,R\varphi)} F}.
\end{equation}
Using the second Griffiths inequality, we deduce that
\begin{equation}
\ddp{}{t} \E_{t}[e^{(f,\varphi)}] = \cov_t((\varphi,R\varphi),e^{(f,\varphi)}) \geq 0,
\end{equation}
where $\E_t$ denotes the expectation of the measure defined as in  \eqref{e:infvol-Lambda} but with interpolated Gaussian covariance,
and both functions inside the two arguments of the covariance are absolutely monotone.
This implies monotonicity in $M$ and thus existence of the limit in distribution as well as the convergence of $\phi(f)$ in law for $f\in L^2$.

\addtocontents{toc}{\protect\setcounter{tocdepth}{1}}

\appendix
\section{Hierarchical covariance and branching random walk}
\label{app:hier}

In this appendix, we collect properties of the hierarchical Laplacian and its inverse that we need,
see also \cite{MR3969983}.
The block spin projections $Q_j$ have kernel entries
\begin{equation} \label{e:Q-entries}
  Q_j(x,y) = L^{-dj} \1_{d_H(x,y) \leq L^j}  = L^{-2j} \1_{d_H(x,y) \leq L^j} = L^{-2j} \1_{j(x,y) \leq j},
\end{equation}
with respect to Lebesgue measure so that
\begin{equation}
  Q_j f(x) = \int Q_j(x,y) f(y)\, dy
\end{equation}
is the orthogonal projection with respect to the $L^2$ inner product on $\R^2$
onto the space $X_j$ of functions constant on the $j$-blocks.
The hierarchical spectral projections are then defined by
\begin{equation}
  \label{e:hier-proj}
  P_j= Q_{j-1}-Q_j.
\end{equation}
The projections $(P_k)$ and $(Q_j)$ satisfy $P_k Q_j= (Q_{k-1}-Q_k)Q_j = Q_{j\vee (k-1)} - Q_{j\vee k}$ and therefore
\begin{equation}
 P_{k}Q_j = 0 \quad (j\geq k), \qquad P_{k}Q_j = P_{k} \quad (j\leq k-1).
\end{equation}
In particular, for any $j<M$, the projections $P_{j+1}, \dots, P_M, Q_M$ are orthogonal and
\begin{equation}
 Q_j = P_{j+1}+ ... + P_M + Q_M,
\end{equation}
and if $j={-N}$ then $Q_{-N}$ is the identity on the space $X_{-N}$ of functions on $\Lambda_{-N}$ so that the
right-hand side provides an orthogonal resolution of the identity.
The hierarchical covariance is defined by
\begin{equation}
  C_{-N,M} = \sum_{k=-N}^M \gamma_k Q_k, \qquad \gamma_k= L^{2k} \log L.
\end{equation}
Define
\begin{equation}
  \gamma_{j,k-1}
  = \sum_{l=j}^{k-1} \gamma_l 
  \approx L^{2k}
  .
\end{equation}

\begin{lemma} \label{lem:hiergen}
  For scales $j<M$, the truncated hierarchical covariance $C_{j,M} = \sum_{k=j}^M \gamma_k Q_k$ 
  restricted to the space $X_j$ of block-constant functions is the inverse of
  $-\Delta_{H,j,M}+ 1/\gamma_{j,M}$, where $\Delta_{H,j,M}$ is the generator
  of a Markov process on $\cB_j$. 
  Moreover, for any $m^2 > 0$,
  \begin{equation} \label{e:DeltaHj-DeltaH}
   (-\Delta_{H,j,M}+1/\gamma_{j,M}+m^2)^{-1}(B,B') \lesssim (-\Delta_{H}+m^2)^{-1}(B,B')
  \end{equation}
  for distinct $j$-blocks $B,B'$ and,
  for all $j$-blocks $B,B'$,
  \begin{equation} \label{e:CMM+1}
      C_{j,M+1}^{-1}(B,B') \leq C_{j,M}^{-1}(B,B').
  \end{equation}
\end{lemma}

\begin{proof}
It suffices to consider the claim for $j=0$.
Using \eqref{e:hier-proj} the given covariance is
\begin{equation}
  C = \sum_{k=0}^M \gamma_kQ_k
  = \sum_{k=0}^M \gamma_k \sum_{l=k+1}^M P_l + \sum_{k=0}^M \gamma_k Q_M
  = \sum_{k=1}^M  { \gamma_{0,k-1}} P_k + { \gamma_{0,M} } Q_M.
\end{equation}
Since this is a spectral decomposition, the inverse is
\begin{align}
  \cL = C^{-1} = \sum_{k=1}^M { \gamma_{0,k-1}^{-1}} P_k + { \gamma_{0,M}^{-1} } Q_M
  &= \sum_{k=1}^M { \gamma_{0,k-1}^{-1} } (Q_{k-1}-Q_k) + { \gamma_{0,M}^{-1} } Q_M
    \nnb
  & = \sum_{k=0}^{M-1} { \gamma_{0,k}^{-1} } Q_{k}- \sum_{k=1}^M { \gamma_{0,k-1}^{-1} }  Q_k + { \gamma_{0,M}^{-1} } Q_M
    \nnb
  & = { \gamma_{0}^{-1}} Q_0+  \sum_{k=1}^{M-1} { (\gamma_{0,k}^{-1}-\gamma_{0,k-1}^{-1})} Q_{k} - { \gamma_{0,M-1}^{-1}} Q_M + { \gamma_{0,M}^{-1} } Q_M
    \nnb
  & = { \gamma_{0}^{-1} } Q_0+  \sum_{k=1}^{M} { (\gamma_{0,k}^{-1}-\gamma_{0,k-1}^{-1})} Q_{k}
    .
    \label{e:hiergen-pf-C-1}
\end{align}

To show that $-\cL$ is the defective generator of a random walk, we need to show that $\cL_{xx} \geq 0$ and $\cL_{xy} \leq 0$ for $x\neq y$
and $-\sum_{y: y\neq x} \cL_{xy} \leq \cL_{xx}$.
Since the diagonal entries of $Q_k$ are $L^{-2k}$, the diagonal entries of $\cL$ are
\begin{equation}
  \cL_{xx}
  = { \gamma_{0}^{-1}} + \sum_{k=1}^M { (\gamma_{0,k}^{-1}-\gamma_{0,k-1}^{-1})}L^{-2k}
  = (L^2-1) \sum_{k=1}^M { \gamma_{0,k-1}^{-1}} L^{-2k} + { \gamma_{0,M}^{-1} }L^{-2M}>0.
\end{equation}
Using \eqref{e:Q-entries} for the entries of $Q_k$ and $\gamma_{0,k} > \gamma_{0,k-1}$,
the offdiagonal entries of $\cL$ are
\begin{equation}
  \cL_{xy}
  = \sum_{{ k={j(x,y)\vee 1}}}^M { (\gamma_{0,k}^{-1}-\gamma_{0,k-1}^{-1}) }L^{-2k} \leq 0,
\end{equation}
where $j(x,y)$ is the smallest scale such that $x,y$ are in the same block.
Since $\cL$ is positive definite and $1$ is an eigenvector, also $\sum_{y} \cL_{xy} = (\cL 1)_x = { \gamma_{0,M}^{-1}} >0$.
Define
\begin{align}
  -\Delta_{H,0,M} = \cL - { \gamma_{0,M}^{-1}} \id.
\end{align}
Then $\Delta_{H,0,M}$ is the generator of a Markov process. 

The pointwise comparison \eqref{e:DeltaHj-DeltaH} follows from the spectral representations of the two resolvents
and the positivity of the hierarchical kernels. Indeed,
\begin{equation} \label{e:DeltaHjm}
  (C^{-1}+m^2)^{-1}
  = \sum_{k=1}^M (\gamma_{0,k-1}^{-1}+m^2)^{-1} P_k + (\gamma_{0,M}^{-1}+m^2)^{-1} Q_M
  = \sum_{k=0}^M \eta_{0,k}(m^2) Q_k,
\end{equation}
and, on the other hand,
\begin{equation} \label{e:DeltaHm}
  (-\Delta_H + m^2)^{-1} = \sum_{k\in \Z} (\gamma_{-\infty,k-1}^{-1}+m^2)^{-1} P_k
  = \sum_{k\in \Z} \eta_{-\infty,k}(m^2) Q_k
\end{equation}
where $\eta_{0,0}(m^2) = (\gamma_0^{-1}+m^2)^{-1}$ and, for $k>j$,
\begin{equation}
  \eta_{j,k}(m^2)
  = (\gamma_{j,k}^{-1}+m^2)^{-1} - (\gamma_{j,k-1}^{-1}+m^2)^{-1}
  = \frac{\gamma_k}{(1+\gamma_{j,k}m^2)(1+\gamma_{j,k-1}m^2)}
  .
\end{equation}
Thus when $m^2=0$ then $\eta_{j,k} = \gamma_{j,k}-\gamma_{j,k-1} = \gamma_k$ as expected and, for $k>0$,
\begin{equation}
  \frac{\eta_{0,k}(m^2)}{\eta_{-\infty,k}(m^2)}
  \leq 
  \frac{\gamma_{-\infty,k}\gamma_{-\infty,k-1}}
  {\gamma_{0,k}\gamma_{0,k-1}} \lesssim 1.
\end{equation}
For $d_H(x,y) > 1$, only terms $k> 0$ contribute in both \eqref{e:DeltaHjm} and \eqref{e:DeltaHm} and the claim follows.

The inequality \eqref{e:CMM+1} is immediate from \eqref{e:hiergen-pf-C-1}
since $\gamma_{0,k}^{-1}-\gamma_{0,k-1}^{-1} \leq 0$ and $Q_k$ has nonnegative entries.
\end{proof}

\begin{corollary}
  The regularized sinh-Gordon measure $\nu$ 
  as well as    the renormalized measures $\nu_j$ satisfy the FKG and first and second Griffiths inequalities.
\end{corollary}

\begin{proof}
  This follows from Lemma~\ref{lem:hiergen} which shows that the Gaussian part
  of the measure is ferromagnetic, together with the fact that the potential
  (respectively renormalized potential) is local.
\end{proof}

\section{Discussion of relation to controversies in physics and questions}

\subsection{Comparison of perturbation theory in sine- and sinh-Gordon models}
\label{app:sinsinhpert}

Continuing the discussion from Section~\ref{sec:interpretation},
consider the second order term in the perturbation theory for the renormalized potential
in the sine- and sinh-Gordon model.

For the sine-Gordon model in volume $\Lambda$ and parameters $(\beta,z)$, this term is given by
\begin{align}
  &\frac{z^2}{2} \int_{(\Lambda \times \{\pm 1\})^2} d\xi_1 \, d\xi_2 \, \Eg\qa{\wick{e^{i2\beta \zeta_1 \varphi(x_1)}}\wick{e^{i2\beta \zeta_2 \varphi(x_2)}}} e^{i2\beta\sum_j \sigma_j \varphi(x_j)}
    \nnb
  &\approx
    \frac{z^2}{2} \int_{(\Lambda \times \{\pm 1\})^2} d\xi_1 \, d\xi_2 \, |x_1-x_2|^{4\beta^2 \sigma_1\sigma_2} e^{i2\beta\sum_j \sigma_j \varphi(x_j)}, \qquad \int d\xi_i = \int dx_i \sum_{\sigma_i= \pm 1}
    ,
\end{align}
where $\Eg[\cdot]$ is the expectation of a large-distance regularized Gaussian free field $\zeta$ so that the
exponential two-point function is
at short distances given by $\approx |x_1-x_2|^{4\beta^2\sigma_1\sigma_2}$.
(See for example \cite[(4.15)]{MR4767492} or \cite[(5.17)]{2508.14806}, where
the convention for $\beta$ is different: our $\beta^2$ corresponds to $\beta/8\pi$
in these references.)
For $\varphi =0$ the last integral is divergent when $\beta^2 \geq 1/2$.
It is the ``neutral'' term $\sigma_1\neq \sigma_2$ that makes the integral divergent,
and as a consequence the difference between this term with a smooth $\varphi$ and with $\varphi=0$
is well behaved and the divergence can therefore be absorbed in a counterterm (infinite energy renormalization).
For example,
\begin{equation}
  \frac{z^2}{2} \int_{(\Lambda \times \{\pm 1\})^2} d\xi_1 \, d\xi_2 \, \Eg\qa{\wick{e^{i2\beta \sigma_1 \zeta(x_1)}}\wick{e^{i2\beta \sigma_2 \zeta(x_2)}}} (e^{i2\beta\sum_j \sigma_j \varphi(x_j)}-1) = O(z^2) \|\nabla \varphi\|_{L^\infty}.
\end{equation}

The situation is different for the sinh-Gordon model.
The second order term in perturbation theory is given by the same expression with $\beta^2 = -b^2$ and $z=-\mu$:
\begin{align}
  &\frac{\mu^2}{2} \int_{(\Lambda \times \{\pm 1\})^2} d\xi_1 \, d\xi_2 \, \Eg\qa{\wick{e^{2b \sigma_1 \zeta(x_1)}}\wick{e^{2b \sigma_2 \zeta(x_2)}}} e^{2b\sum_j \sigma_j \varphi(x_j)}
  \nnb
  &\approx
  \frac{\mu^2}{2} \int_{(\Lambda \times \{\pm 1\})^2} d\xi_1 \, d\xi_2 \, |x_1-x_2|^{-4b^2 \sigma_1\sigma_2} e^{2b\sum_j \sigma_j \varphi(x_j)}.
\end{align}
For $\varphi=0$ the integral is again divergent when $b^2 \geq 1/2$,
but differently from the situation of the sine-Gordon model,
the divergence cannot be regularized by a constant counterterm.
It is now the ``charged'' term $\sigma_1=\sigma_2$ that causes the divergence
and smoothness of $\varphi$ does not help.

From a perturbation theory point of view, the fact that we still obtain an estimate with error $O(\mu^p)$ with $p \in (1,2)$ when $b^2 \geq 1/2$ in \eqref{e:intro-HeL1} is thus a nonperturbative effect.
  

\subsection{Relation to controversies in physics and questions}
\label{app:physics}

The predicted behavior of the sinh-Gordon model is surprisingly controversial in physics.
We hope that our rigorous analysis of the hierarchical version of the model can provide a starting point
for a resolution.
In this appendix, we collect some natural questions whose answer we think would clarify the picture,
and which we expect should be addressed in the hierarchical version of the model first.

The consensus seems to be that the predictions for $b^2 \in (0,1/2)$ are relatively trustworthy.
From a mathematical perspective, however, the analytic continuation seems completely unclear.

\begin{question} \label{q:1}
  Do the correlation functions defined from the infinite-volume measure $\nu^{\hShG(b,\mu)}$ extend to analytic functions in a complex neighborhood containing
  $b^2 \in (0,1/2)$? Can they be analytically continued to $\beta^2 = -b^2 \in (0,1/2)$?
\end{question}

The first puzzle from the perspective of physics concerns the region $b^2 \in [1/2,1)$. The comparison between sine- and sinh-Gordon models
refers to the particle spectrum of the sine-Gordon theory (in particular the soliton-antisoliton bound states -- known as breathers)
which is expected to change at $b^2=1/2$.
Indeed, it is expected that the number of breathers in the massless sine-Gordon model is given by
(with $[\cdot]$ denoting the integer part):
\begin{equation}
  \qa{\frac{1-b^2}{b^2}}.
\end{equation}
This leads to doubts about the validity of the conjectures for $b^2 \in [1/2,1)$, see \cite[Section~2.2.7]{MR4258290}.

Probabilistically, the regime $b^2 \in [1/2,1)$ is technically more difficult in our analysis to handle because second-order perturbation theory does not exist
(related to the fact that the GMC is not in $L^2$), but we see no essential physical difference compared to the region $b^2 \in (0,1/2)$.
In particular, in the hierarchical model, we show that there is a strictly positive physical mass for all $b^2 \in (0,1)$.

\begin{question} \label{q:2}
  Are the (infinite-volume massless) correlation functions analytic in $b^2 \in (0,1)$?
\end{question}


Given that the physical mass is strictly positive for $b\in (0,1)$,
at least in the hierarchical version of the sinh-Gordon model,
it is natural to ask for its asymptotics as $b\uparrow 1$.
For the Euclidean model, it is conjectured \cite[(2.27)]{MR4258290} that the mass behaves as $(1-b)^{1/4}$ as $b\uparrow 1$.
Our analysis does not provide information about the $b\uparrow 1$ asymptotics and it is unclear to us if one should expect
the same asymptotics in the hierarchical model.

\begin{question} \label{q:4}
  What are the asymptotics of the mass as $b\uparrow 1$?
\end{question}

For $b^2 =1$ and $b^2 \in (1,\infty)$ it would be interesting to understand what exactly happens.
Locally, there is a nontrivial limit of the regularized sinh-Gordon measure when the multiplicative counterterm
$\mu_\epsilon$ is chosen as in the normalization of the critical Gaussian multiplicative chaos (GMC)
or supercritical GMC, see \cite{MR3274356}.
The most likely scenario consistent with the conjectures from physics is that this limit is massless,
for example in the sense that $L^{2j}\He V_j$ not being bounded away from $0$ for $j\approx 0$,
and that the infinite volume limit does not exist.

\begin{question} \label{q:5}
  Is the hierarchical sinh-Gordon model for $b^2 \in [1,\infty)$ defined in terms of the critical and supercritical GMC massless
  in the sense that $L^{2j}\He V_j(0)$ is not bounded away from $0$?
\end{question}

For the Euclidean model there is a controversy in physics
related to the correct treatment of the zero mode, discussed in particular in \cite{MR4258290,MR4430201}.
The exact prediction for the $S$-matrix of the sinh-Gordon model depends on $b$ through $b+1/b$ and an initial proposal
was a ``strong-weak'' duality $b \leftrightarrow 1/b$.
From the path integral perspective, this seems basically impossible -- a fact well recognized in the physics literature.
There are various related situations such as the Fyodorov--Le Doussal--Rosso conjecture \cite{MR2882779} where such a duality is suggested
but convincing evidence instead supports a freezing transition.
Bernard--LeClair proposed the possibility that the massless theory describing the sinh-Gordon model for $b^2>1$
is not a free field but a nontrivial flow between two conformal field theories \cite{MR4430201}.
This proposal is based on the ``exact'' renormalization group beta functions proposed in \cite{PhysRevLett.86.4753,zbMATH01617535,MR4430201},
which cover both the sine- and sinh-Gordon models in a unified fashion. These beta functions are derived for a class of models
that is expected to include the sine-Gordon model, and the sinh-Gordon version is essentially obtained by the replacement $\beta = ib$.
Based on the differences of the perturbative structure of the sine- and sinh-Gordon models, such as the divergence of the second order term,
we are not convinced this procedure is legitimate.

\section*{Acknowledgements}

RB thanks Denis Bernard, Thierry Bodineau, Benoit Dagallier, Karol Kozlowski, Antti Kupiainen, and Christian Webb for various
discussions related to the topic of the paper.

OA and RB were partially supported by NSF grant DMS-2348045 and the Simons Collaboration
grant on Probabilistic Paths to Quantum Field Theory.
MH was supported by the FWF SFB F 1002 grant
Discrete Random Structures.
OZ was supported by the Israel science foundation grant \#615/24.

\section*{AI declaration}

OA and RB were supported in part by a grant of access to OpenAI models through the ChatGPT for Academic Researchers program.
This resource was only used for checking spelling and for typos.

\bibliography{all}

@misc{biermé2017generalizedrandomfieldslevys,
      title={Generalized random fields and L\'evy's continuity theorem on the space of tempered distributions}, 
      author={Hermine Biermé and Olivier Durieu and Yizao Wang},
      year={2017},
      eprint={1706.09326},
      archivePrefix={arXiv},
      primaryClass={math.PR},
      url={https://arxiv.org/abs/1706.09326}, 
}

@Article{MR648362,
    Author = {Brydges, D. and Fr{\"o}hlich, J. and Spencer, T.},
    Title = "The random walk representation of classical spin systems and correlation inequalities",
    Journal = "Commun. Math. Phys.",
    Year = "1982",
    Number = "1",
    Pages = "123--150",
    Volume = "83",
    Url = "https://projecteuclid.org/getRecord?id=euclid.cmp/1103920749"
}

@InCollection{MR2523458,
    author = "Brydges, D.C.",
    title = "Lectures on the renormalisation group",
    booktitle = "Statistical mechanics",
    publisher = "Amer. Math. Soc.",
    year = "2009",
    volume = "16",
    series = "IAS/Park City Math. Ser.",
    pages = "7--93",
}

@Article{MR693402,
    author = "Gaw{\polhk{e}}dzki, K. and Kupiainen, A.",
    title = "Triviality of {$\varphi ^{4}_{4}$}\ and all that in a hierarchical model approximation",
    journal = "J. Statist. Phys.",
    year = "1982",
    volume = "29",
    number = "4",
    pages = "683--698",
    url = "https://doi.org/10.1007/BF01011785"
}

@Article{MR859369,
    author = "Marchetti, D.H.U. and Perez, J.F.",
    title = "A hierarchical model exhibiting the {K}osterlitz-{T}houless fixed point",
    journal = "Phys. Lett. A",
    year = "1986",
    volume = "118",
    number = "2",
    pages = "74--76",
    url = "https://doi.org/10.1016/0375-9601(86)90650-X"
}

@Book{MR3155209,
    author = "Bakry, D. and Gentil, I. and Ledoux, M.",
    title = "Analysis and geometry of {M}arkov diffusion operators",
    publisher = "Springer, Cham",
    year = "2014",
    volume = "348",
    series = "Grundlehren der Mathematischen Wissenschaften",
    isbn = "978-3-319-00226-2; 978-3-319-00227-9",
    doi = "10.1007/978-3-319-00227-9",
    pages = "xx+552",
    url = "https://doi.org/10.1007/978-3-319-00227-9"
}

@Book{MR887102,
    author = "Glimm, J. and Jaffe, A.",
    publisher = "Springer-Verlag",
    title = "Quantum physics",
    year = "1987",
    edition = "Second",
    isbn = "0-387-96476-2",
    note = "A functional integral point of view",
    pages = "xxii+535",
}

@Book{MR0489552,
    author = "Simon, B.",
    publisher = "Princeton University Press",
    title = "The {$P(\phi )_{2}$} {E}uclidean (quantum) field theory",
    year = "1974",
    note = "Princeton Series in Physics",
    pages = "xx+392",
}

@Article{MR4061408,
    author = "Bauerschmidt, R. and Bodineau, T.",
    journal = "Commun. Math. Phys.",
    title = "Spectral {G}ap {C}ritical {E}xponent for {G}lauber {D}ynamics of {H}ierarchical {S}pin {M}odels",
    year = "2020",
    number = "3",
    pages = "1167--1206",
    volume = "373",
    doi = "10.1007/s00220-019-03553-x",
    eprint = "1809.02075",
    url = "https://doi.org/10.1007/s00220-019-03553-x"
}

@Book{MR3969983,
    author = "Bauerschmidt, R. and Brydges, D.C. and Slade, G.",
    publisher = "Springer, Singapore",
    title = "Introduction to a renormalisation group method",
    year = "2019",
    isbn = "978-981-32-9591-9; 978-981-32-9593-3",
    series = "Lecture Notes in Mathematics",
    volume = "2242",
    doi = "10.1007/978-981-32-9593-3",
    eprint = "1907.05474",
    pages = "xii+281",
    url = "https://doi.org/10.1007/978-981-32-9593-3"
}

@Article{MR3274356,
    author = "Rhodes, R. and Vargas, V.",
    journal = "Probab. Surv.",
    title = "Gaussian multiplicative chaos and applications: a review",
    year = "2014",
    pages = "315--392",
    volume = "11",
    doi = "10.1214/13-PS218",
    url = "https://doi.org/10.1214/13-PS218"
}

@article{MR395659,
    AUTHOR = "Ellis, R.S. and Monroe, J.L. and Newman, C.M.",
    TITLE = "The {GHS} and other correlation inequalities for a class of even ferromagnets",
    JOURNAL = "Commun. Math. Phys.",
    VOLUME = "46",
    YEAR = "1976",
    NUMBER = "2",
    PAGES = "167--182",
    URL = "http://projecteuclid.org/euclid.cmp/1103899588"
}

@article{MR4303014,
    AUTHOR = "Bauerschmidt, R. and Bodineau, T.",
    TITLE = "Log-{S}obolev inequality for the continuum sine-{G}ordon model",
    JOURNAL = "Comm. Pure Appl. Math.",
    VOLUME = "74",
    YEAR = "2021",
    NUMBER = "10",
    PAGES = "2064--2113",
    DOI = "10.1002/cpa.21926",
    URL = "https://doi.org/10.1002/cpa.21926",
    eprint = "1907.12308",
}

@Article{MR4767492,
    author = "Bauerschmidt, R. and Webb, C.",
    journal = "J. Eur. Math. Soc. (JEMS)",
    title = "The {C}oleman correspondence at the free fermion point",
    year = "2024",
    number = "9",
    pages = "3137--3241",
    volume = "26",
    doi = "10.4171/jems/1329",
    eprint = "2010.07096",
    url = "https://doi.org/10.4171/jems/1329"
}

@Article{MR4798104,
    author = "Bauerschmidt, R. and Bodineau, T. and Dagallier, B.",
    journal = "Probab. Surv.",
    title = "Stochastic dynamics and the {P}olchinski equation: {A}n introduction",
    year = "2024",
    pages = "200--290",
    volume = "21",
    doi = "10.1214/24-ps27",
    url = "https://doi.org/10.1214/24-ps27",
    eprint = "2307.07619"
}

@article{MR4258290,
    AUTHOR = "Konik, R. and L\'ajer, M. and Mussardo, G.",
    TITLE = "Approaching the self-dual point of the sinh-{G}ordon model",
    JOURNAL = "J. High Energy Phys.",
    YEAR = "2021",
    NUMBER = "1",
    PAGES = "Paper No. 014, 82",
    DOI = "10.1007/jhep01(2021)014",
    URL = "https://doi.org/10.1007/jhep01(2021)014"
}

@Article{2508.14806,
    author = "Bauerschmidt, R. and Mason, S. and Webb, C.",
    title = "Twisted {D}irac operators and fractional correlations of the massless sine-{G}ordon model at the free fermion point",
    archiveprefix = "arXiv",
    eprint = "2508.14806",
    note = "Preprint, arXiv:2508.14806"
}

@article{MR4430201,
    AUTHOR = "Bernard, D. and LeClair, A.",
    TITLE = "The sinh-{G}ordon model beyond the self dual point and the freezing transition in disordered systems",
    JOURNAL = "J. High Energy Phys.",
    YEAR = "2022",
    NUMBER = "5",
    PAGES = "Paper No. 022, 23",
    DOI = "10.1007/jhep05(2022)022",
    URL = "https://doi.org/10.1007/jhep05(2022)022"
}

@incollection{MR4680395,
    AUTHOR = "Kozlowski, K.K.",
    TITLE = "Bootstrap approach to {$1+1$}-dimensional integrable quantum field theories: the case of the sinh-{G}ordon model",
    BOOKTITLE = "I{CM}---{I}nternational {C}ongress of {M}athematicians. {V}ol. 5. {S}ections 9--11",
    PAGES = "4096--4118",
    PUBLISHER = "EMS Press, Berlin",
    YEAR = "[2023] \copyright 2023",
    ISBN = "978-3-98547-063-1; 978-3-98547-563-6; 978-3-98547-058-7",
}

@article{MR4607722,
    AUTHOR = "Kozlowski, K.K.",
    TITLE = "On convergence of form factor expansions in the infinite volume quantum {S}inh-{G}ordon model in {$1+1$} dimensions",
    JOURNAL = "Invent. Math.",
    VOLUME = "233",
    YEAR = "2023",
    NUMBER = "2",
    PAGES = "725--827",
    DOI = "10.1007/s00222-023-01192-7",
    URL = "https://doi.org/10.1007/s00222-023-01192-7"
}

@article{MR356761,
    AUTHOR = "Albeverio, S. and H{\o}egh-Krohn, R.",
    TITLE = "The {W}ightman axioms and the mass gap for strong interactions of exponential type in two-dimensional space-time",
    JOURNAL = "J. Func. Anal.",
    VOLUME = "16",
    YEAR = "1974",
    PAGES = "39--82",
    DOI = "10.1016/0022-1236(74)90070-6",
    URL = "https://doi.org/10.1016/0022-1236(74)90070-6"
}

@article{zbMATH01617535,
    author = "Bernard, D. and LeClair, A.",
    title = "Strong-weak coupling duality in anisotropic current interactions",
    journal = "Phys. Lett., B",
    volume = "512",
    number = "1-2",
    pages = "78--84",
    year = "2001",
    language = "English",
    doi = "10.1016/S0370-2693(01)00695-5",
    zbMATH = "1617535",
    Zbl = "0969.81595"
}

@Article{PhysRevLett.86.4753,
    author = "Gerganov, B. and LeClair, A. and Moriconi, M.",
    journal = "Phys. Rev. Lett.",
    title = "{B}eta {F}unction for {A}nisotropic {C}urrent {I}nteractions in {2D}",
    year = "2001",
    month = "May",
    pages = "4753--4756",
    volume = "86",
    doi = "10.1103/PhysRevLett.86.4753",
    issue = "21",
    numpages = "0",
    publisher = "American Physical Society",
    url = "https://link.aps.org/doi/10.1103/PhysRevLett.86.4753"
}

@Article{2408.16574,
    author = "Barashkov, N. and Oikarinen, J. and Wong, M.D.",
    title = "Small deviations of {G}aussian multiplicative chaos and the free energy of the two-dimensional massless {S}inh--{G}ordon model",
    year = "2024",
    eprint = "2408.16574",
    publisher = "arXiv",
    note = "Preprint, arXiv:2408.16574"
}

@article{MR5055747,
    AUTHOR = "Guillarmou, C. and Gunaratnam, T.S. and Vargas, V.",
    TITLE = "2d {S}inh-{G}ordon {M}odel on the {I}nfinite {C}ylinder",
    JOURNAL = "Commun. Math. Phys.",
    VOLUME = "407",
    YEAR = "2026",
    NUMBER = "5",
    PAGES = "Paper No. 97",
    DOI = "10.1007/s00220-026-05588-3",
    URL = "https://doi.org/10.1007/s00220-026-05588-3"
}

@Article{2408.16649,
    author = "Hofstetter, M. and Zeitouni, O.",
    title = "Decay of correlations for the massless hierarchical {L}iouville model in infinite volume",
    year = "2026",
    journal = "Commun. Math. Phys (to appear)",
}

@article{MR4528973,
    AUTHOR = "Hoshino, M. and Kawabi, H. and Kusuoka, S.",
    TITLE = "Stochastic quantization associated with the {$\exp(\Phi)_2$}-quantum field model driven by space-time white noise on the torus in the full {$L^1$}-regime",
    JOURNAL = "Probab. Theory Related Fields",
    VOLUME = "185",
    YEAR = "2023",
    NUMBER = "1-2",
    PAGES = "391--447",
    DOI = "10.1007/s00440-022-01126-z",
    URL = "https://doi.org/10.1007/s00440-022-01126-z"
}

@article{MR2642887,
    AUTHOR = "Robert, R. and Vargas, V.",
    TITLE = "Gaussian multiplicative chaos revisited",
    JOURNAL = "Ann. Probab.",
    VOLUME = "38",
    YEAR = "2010",
    NUMBER = "2",
    PAGES = "605--631",
    DOI = "10.1214/09-AOP490",
    URL = "https://doi.org/10.1214/09-AOP490"
}

@article{MR3269690,
    AUTHOR = "Gallavotti, G.",
    TITLE = "Renormalization group and divergences",
    JOURNAL = "J. Stat. Phys.",
    VOLUME = "157",
    YEAR = "2014",
    NUMBER = "4-5",
    PAGES = "743--754",
    DOI = "10.1007/s10955-014-1027-6",
    URL = "https://doi.org/10.1007/s10955-014-1027-6"
}

@article{MR829798,
    AUTHOR = "Kahane, J.-P.",
    TITLE = "Sur le chaos multiplicatif",
    JOURNAL = "Ann. Sci. Math. Qu\'ebec",
    VOLUME = "9",
    YEAR = "1985",
    NUMBER = "2",
    PAGES = "105--150",
}

@article{MR4683324,
    AUTHOR = "Chewi, S. and Pooladian, A.-A.",
    TITLE = "An entropic generalization of {C}affarelli's contraction theorem via covariance inequalities",
    JOURNAL = "C. R. Math. Acad. Sci. Paris",
    VOLUME = "361",
    YEAR = "2023",
    PAGES = "1471--1482",
    DOI = "10.5802/crmath.486",
    URL = "https://doi.org/10.5802/crmath.486"
}

@incollection{MR3497718,
    AUTHOR = "Heurteaux, Y.",
    TITLE = "An introduction to {M}andelbrot cascades",
    BOOKTITLE = "New trends in applied harmonic analysis",
    SERIES = "Appl. Numer. Harmon. Anal.",
    PAGES = "67--105",
    PUBLISHER = {Birkh\"auser/Springer, Cham},
    YEAR = "2016",
    ISBN = "978-3-319-27871-1; 978-3-319-27873-5",
}

@article{MR431355,
    AUTHOR = "Kahane, J.-P. and Peyri\`ere, J.",
    TITLE = "Sur certaines martingales de {B}enoit {M}andelbrot",
    JOURNAL = "Advances in Math.",
    VOLUME = "22",
    YEAR = "1976",
    NUMBER = "2",
    PAGES = "131--145",
    DOI = "10.1016/0001-8708(76)90151-1",
    URL = "https://doi.org/10.1016/0001-8708(76)90151-1"
}

@article{MR2882779,
    AUTHOR = "Fyodorov, Y.V. and Le Doussal, P. and Rosso, A.",
    TITLE = "Statistical mechanics of logarithmic {REM}: duality, freezing and extreme value statistics of {$1/f$} noises generated by {G}aussian free fields",
    JOURNAL = "J. Stat. Mech. Theory Exp.",
    YEAR = "2009",
    NUMBER = "10",
    PAGES = "P10005, 32",
    DOI = "10.1088/1742-5468/2009/10/p10005",
    URL = "https://doi.org/10.1088/1742-5468/2009/10/p10005"
}

@Article{2512.18927,
  author        = {Kusuoka, Seiichiro and Nagoji, Hirotatsu},
  title         = {Stochastic quantization of the weighted exponential {QFT}},
  year          = {2025},
  archiveprefix = {arXiv},
  eprint        = {2512.18927},
  publisher     = {arXiv},
}

@Article{2502.02554,
    author = "Garban, C. and Kupiainen, A.",
    title = "Energy field of critical {I}sing model and examples of singular fields in {QFT}",
    year = "2026",
    eprint = "2502.02554",
    publisher = "arXiv",
    note = "Preprint, arXiv:2502.02554"
}

@article{Vilas,
title = "Renormalizing small ball events for branching random walk",
author = "V. Winstein",
note = "In preparation",
}
\bibliographystyle{plain}

\end{document}